\documentclass[11pt]{article}
\usepackage{amsmath}
\usepackage{amsthm}
\usepackage{amsfonts}
\usepackage{amssymb}
\usepackage{latexsym}
\usepackage{diagbox}
\usepackage{mathtools}

\usepackage{array}

\usepackage{subcaption}
\usepackage{booktabs}
\usepackage{tkz-graph}
\usepackage{graphics,graphicx}
\usepackage{multirow}

\usepackage{tikz}
\usepackage{tikz-cd}

\usetikzlibrary{matrix,arrows}
\usetikzlibrary{fit}
\usetikzlibrary{patterns}
\usepackage{bm}
\usetikzlibrary{chains,fit,shapes}
\usetikzlibrary{positioning,calc}
\usetikzlibrary{graphs}
\usetikzlibrary{arrows,decorations.markings}
\usepackage[normalem]{ulem} 
\usepackage{hyperref}
\hypersetup{
  colorlinks=true,
  linkcolor=black,
  citecolor=black,
  urlcolor=black
}

\newtheorem{thm}{Theorem}[section]  
\newtheorem{lem}[thm]{Lemma}      
\newtheorem{prop}[thm]{Proposition}
\newtheorem{cor}[thm]{Corollary}

\usepackage{listings}
\usepackage{xcolor}

\begin{document}
\begin{center}
{\large \bf Wilf Equivalence for Length-Three Patterns and Flat POPs, and a Conjecture of Qiu and Remmel}
\end{center}
\begin{center}
Shiqi Cao$^{1}$, Sergey Kitaev$^{2}$ and Yuxin Wu$^{3}$\\[6pt]

$^{1}$Center for Combinatorics, LPMC, Nankai University, Tianjin 300071, P. R. China

$^{2}$Department of Mathematics and Statistics, University of Strathclyde, 26 Richmond Street, Glasgow G1 1XH, United Kingdom

$^{3}$School of Mathematical Sciences, LPMC, Nankai University, Tianjin 300071, P.R. China.
 Email: $^{1}${\tt shiqicao@mail.nankai.edu.cn},
       $^{2}${\tt sergey.kitaev@strath.ac.uk} 
       $^{3}${\tt yuxinwu@mail.nankai.edu.cn}
\end{center}

\noindent\textbf{Abstract.}
It is well known that, for each
classical pattern $\tau$ of length 3, the number of $\tau$-avoiding permutations
of length $n$ is the $n$th Catalan number, and numerous bijections between
different length-three avoidance classes have been constructed and studied.
In this paper, we refine this classical problem by studying Wilf equivalence
among permutations that simultaneously avoid a classical pattern of length
three and a flat partially ordered pattern.

Partially ordered patterns (POPs) provide a flexible framework for encoding
families of classical permutation patterns. For $\ell\geq 3$ and
$1\leq x\leq\ell$, let $P_{\ell,x}$ be the length-$\ell$ POP in which the
entry at position $x$ is required to be smaller than all the other entries,
while no relations are imposed among the remaining entries. Such POPs are
called flat POPs. We classify the Wilf equivalences among all pairs
$(\tau,P_{\ell,x})$, where $\tau$ is a classical pattern of length three.
For every $\ell\geq4$, the resulting $6\ell$ pairs form exactly $2\ell-1$
Wilf equivalence classes, while the exceptional case $\ell=3$ gives four
classes. Our proofs combine the derivation of explicit formulas and recurrence relations with the construction of bijections. Moreover, we introduce novel prime-divisor arguments to distinguish the
remaining candidate classes, reducing the problem to showing that a certain
Diophantine equation has no solutions for $\ell\ge 3{,}274$, where the
bound $3{,}274$ is not claimed to be sharp. Finally, by extending our work on POPs, we resolve a conjecture of Qiu and Remmel concerning the distribution of quadrant marked mesh patterns on 132-avoiding permutations and correct an error in their paper that is crucial to the proof. \\
	
	
\section{Introduction}\label{intro-sec} 
Let $[n]:=\{1,2,\ldots,n\}$. A \emph{permutation} of length $n$ is a
rearrangement of the elements of $[n]$, and we denote by $S_n$ the set
of all permutations of $[n]$. 
For a permutation $\pi=\pi_1\pi_2\cdots\pi_n\in S_n$, its \emph{reverse},
\emph{complement} are defined, respectively, by
$\pi^r=\pi_n\pi_{n-1}\cdots\pi_1$, $\pi^c=(n+1-\pi_1)(n+1-\pi_2)\cdots(n+1-\pi_n)$.
Regarding $\pi$ as the bijection $\pi\colon[n]\to[n]$ given by
$\pi(j)=\pi_j$, its \emph{inverse}, denoted by $\pi^{-1}$, is the
unique permutation satisfying $\pi^{-1}(\pi(j))=j$ for every $j\in[n]$.
For example, let $\pi=24153\in S_5$. Then the inverse, reverse, and
complement of $\pi$ are, respectively, $\pi^{-1}=31524$, $\pi^r=35142$, and $\pi^c=42513$.

A permutation $\pi=\pi_1\pi_2\cdots\pi_n\in S_n$ contains a classical pattern
$\tau=\tau_1\tau_2\cdots\tau_k\in S_k$ if there exist indices
$1\leq i_1<i_2<\cdots<i_k\leq n$ satisfying
$\pi_{i_p}<\pi_{i_q}$ if and only if $\tau_p<\tau_q$ for all
$1\leq p,q\leq k$.
Otherwise, $\pi$ avoids $\tau$. 
For a collection $\mathcal P$ of patterns,
we use $\operatorname{Av}_n(\mathcal P)$ to denote the set of permutations in $S_n$
that avoid every member of $\mathcal P$, and set
$s_n(\mathcal P)=|\operatorname{Av}_n(\mathcal P)|$. 
Two collections $\mathcal P$ and $\mathcal Q$ are \emph{Wilf-equivalent} if
$s_n(\mathcal P)=s_n(\mathcal Q)$ for every $n\geq0$.

The reverse, complement, and inverse operations preserve pattern
containment. Consequently, for every $\tau\in S_k$ and
$s\in\{r,c,-1\}$,
\begin{align*}
|\operatorname{Av}_n(\tau)|
   =|\operatorname{Av}_n(\tau^s)|.
\end{align*}
Thus, patterns related by any composition of the reverse, complement, and
inverse operations are Wilf-equivalent. Similar arguments apply to sets of
patterns in place of a single pattern $\tau$.

The avoidance of patterns of length three is a famous result in the theory of permutation patterns: for every $\tau\in S_3$,
$|\operatorname{Av}_n(\tau)|$ is the $n$-th Catalan number. For further details and related bijective proofs, see
\cite{Knuth1969,Knuth1973,Rotem1975}.
Indeed, reverse, complement, and inverse give immediate equivalences within the two symmetry
orbits $\{123,321\}$ and $\{132,213,231,312\}$,
whereas bijections between the two orbits reveal substantially more
structure. 
Many such correspondences have been constructed using Dyck
paths, standard Young tableaux, generating trees and so on.
Claesson and Kitaev surveyed many bijections
between length three avoidance classes, explained how they are related by
the elementary symmetry operations, and compared the permutation statistics
that they preserve \cite{ClaessonKitaev2008,Kitaev2011Patterns}. 
Among these maps, the Simion--Schmidt bijection \cite{SimionSchmidt1985} is especially relevant here: a property of it will
provide one of the principal equivalences used in our classification.

In addition to classical permutation patterns, it is natural to study
\emph{marked mesh patterns}, particularly \emph{quadrant marked mesh patterns}.
The permutation diagram of
$\pi=\pi_1\pi_2\cdots\pi_n\in S_n$ consists of the points
$(p,\pi_p)$, for all $p\in[n]$.
Using any point of the diagram as the origin divides the plane into
four quadrants, numbered I, II, III, and IV in the usual
counterclockwise order. 
A point matches $\operatorname{MMP}(a,b,c,d)$ if the four quadrants satisfy the
conditions specified by $a,b,c$, and $d$, respectively. 
A positive integer parameter requires the corresponding quadrant to contain at
least that many points, a zero parameter imposes no condition, and an
$\emptyset$ parameter requires the corresponding quadrant to contain
no points.
Kitaev, Remmel, and Tiefenbruck initiated a systematic study of these statistics on
$132$-avoiding permutations. 
Their three-part series treats patterns with
one, two, and at least three nonzero parameters, respectively
\cite{KitaevRemmelTiefenbruckI,KitaevRemmelTiefenbruckII,KitaevRemmelTiefenbruckIII}.
The reverse and complement operations transfer these results to the classes
avoiding $231$, $213$, and $312$. 
Qiu and Remmel \cite{QiuRemmel2018} subsequently developed the corresponding distribution theory 
for $123$-avoiding permutations, and hence reversal transfers these results to the class of $321$-avoiding permutations.

For later use, it is convenient to label the points of a permutation
diagram by their values. For $i\in[n]$, let
$\operatorname{Pos}_{\pi}(i)$ denote the position occupied by the entry
$i$ in $\pi$. Thus, the point corresponding to $i$ is
$(\operatorname{Pos}_{\pi}(i),i)$.
Define
\begin{align*}
Q_{\mathrm{II}}^\pi(i)
:=
\#\bigl\{
j\in[n]:
j>i\text{ and }
\operatorname{Pos}_{\pi}(j)<\operatorname{Pos}_{\pi}(i)
\bigr\}.
\end{align*}
Thus, $Q_{\mathrm{II}}^\pi(i)$ is the number of points lying in the
second quadrant relative to
$(\operatorname{Pos}_{\pi}(i),i)$. When the underlying permutation is
clear, we simply write $Q_{\mathrm{II}}(i)$. Similarly, we define $Q_{\mathrm{I}}^\pi(i)$.

Define a \emph{partially ordered pattern ({POP})} $P$ of length $k$ by a $k$-element partially ordered set (poset) labeled by the elements in $\{1,2,\ldots,k\}$.
An occurrence of $P$ in a permutation $\pi_1\cdots\pi_n\in S_n$ is a subsequence $\pi_{i_1}\cdots\pi_{i_k}$, 
where $1\le i_1<\cdots< i_k\le n$, such that $\pi_{i_j}<\pi_{i_m}$ whenever $j<m$ in $P$. 
Thus, a classical pattern of length $k$ corresponds to a $k$-element chain.
POP-avoiding permutations have been studied by many researchers; 
see, for example \cite{BursteinKitaev2008,Kitaev2007,WangYan}.

\begin{figure}[ht]
\centering
\begin{tikzpicture}[scale=1.1]
    \node[circle, fill, inner sep=2pt, label=below:{$x$}]
        (b) at (0,-1.6) {};

    \node[circle, fill, inner sep=2pt] (a1) at (-2.8,0) {};
    \node[circle, fill, inner sep=2pt] (a2) at (-1.1,0) {};
    \node at (0,0) {$\cdots$};
    \node[circle, fill, inner sep=2pt] (a3) at (1.1,0) {};
    \node[circle, fill, inner sep=2pt] (a4) at (2.8,0) {};

    \node at (0,0.65) {$[\ell]\setminus\{x\}$};

    \draw (b) -- (a1);
    \draw (b) -- (a2);
    \draw (b) -- (a3);
    \draw (b) -- (a4);
\end{tikzpicture}
\caption{$P_{\ell,x}$}
\label{fig:Pellx}
\end{figure}

In this paper, we study the family of POPs $P_{\ell,x}$, where
$\ell\geq3$ and $1\leq x\leq\ell$, presented in Figure~\ref{fig:Pellx}.
The upper level of $P_{\ell,x}$ consists of the $\ell-1$ elements in
$[\ell]\setminus\{x\}$.
An occurrence of $P_{\ell,x}$ is an $\ell$-term subsequence whose
$x$-th entry is smaller than all its other entries. 
According to the definition, it is easy to see that for every
$\pi\in S_n$,
\begin{align*}
  \pi\in\operatorname{Av}_n(P_{\ell,x})
  \quad\Longleftrightarrow\quad
  \operatorname{mmp}^{(\ell-x,x-1,0,0)}(\pi)=0,
\end{align*}
where $\operatorname{mmp}^{(a,b,c,d)}(\pi)$ denotes the number of matches of
$\operatorname{MMP}(a,b,c,d)$ in $\pi$.  

Our main result determines all Wilf equivalences among the $6\ell$ pairs
$(\tau,P_{\ell,x})$ with $\tau\in S_3$ and $x\in[\ell]$.
By applying the reverse operation, it suffices to study the three families
$(213,P_{\ell,x})$, $(231,P_{\ell,x})$, and $(321,P_{\ell,x})$.
For $\ell\geq4$, the complete classification is presented in Table~\ref{tab:wilf-classes}, where distinct values of $r$ correspond to distinct classes; in particular, the last row lists singleton classes.

\begin{table}[htbp]
\centering
\renewcommand{\arraystretch}{1.25}
\setlength{\tabcolsep}{5pt}
\setlength{\arrayrulewidth}{0.5pt}
\small

\begin{tabular}{
    |>{\centering\arraybackslash}m{11.5cm}
    |>{\centering\arraybackslash}m{2.8cm}|}
\hline

\textbf{Wilf equivalence class}
&
\textbf{Ref.}
\\
\hline

$\displaystyle
\mathcal E_{\ell}
=
\bigl\{
(321,P_{\ell,1})
\bigr\}$
&
Corollary~\ref{cor-321}
\\
\hline

$\displaystyle
\mathcal A_{\ell,2}
=
\left\{
(321,P_{\ell,2}), (231,P_{\ell,2}),
(231,P_{\ell,1}),
(213,P_{\ell,\ell})
\right\}$
&
\begin{tabular}[c]{@{}c@{}}
Theorem~\ref{231,Pl1}\\
Theorem~\ref{thm_231_method_1}\\
Theorem~\ref{thm:Phi-Psi-preserve-POP}
\end{tabular}
\\
\hline

$\displaystyle
\begin{gathered}
\mathcal A_{\ell,r}
=
\left\{(321,P_{\ell,r}),
(231,P_{\ell,r})\right\}, \ 3\leq r\leq\ell
\end{gathered}$
&
Theorem~\ref{thm:Phi-Psi-preserve-POP}
\\
\hline
$\displaystyle
\begin{gathered}
\mathcal B_{\ell,r}
=
\bigl\{
(213,P_{\ell,r})
\bigr\},\ 1\leq r\leq\ell-1
\end{gathered}$
&
Theorem~\ref{213_no}
\\
\hline
\end{tabular}

\caption{Wilf equivalence classes of the pairs
$(\tau,P_{\ell,x})$ for $\ell\geq4$.}
\label{tab:wilf-classes}
\end{table}

We also prove that no further equivalences exist. 
Consequently, for every $\ell\geq4$, the total number of Wilf equivalence classes is
$1+(\ell-1)+(\ell-1)=2\ell-1$.
When $\ell=3$, it is easy to check that
$\operatorname{Av}_n(213,P_{3,2})
=\operatorname{Av}_n(213,312)$
and
$\operatorname{Av}_n(321,P_{3,3})
=\operatorname{Av}_n(231,321)$.
By Proposition~12 of Simion and Schmidt \cite{SimionSchmidt1985}, together with the standard inverse, complement, and reverse, shows that
$|\operatorname{Av}_n(213,312)|
=
|\operatorname{Av}_n(231,321)|
=
2^{n-1}$.
Consequently, the case $\ell=3$ contains $2\ell-2=4$ Wilf equivalence classes.

By a statistic on a finite set $\mathcal A$, we simply mean a function
$\operatorname{st}:\mathcal A\to\mathbb{Z}_{\geq 0}$. Its distribution over
$\mathcal A$ is encoded by the polynomial $\sum_{\pi\in\mathcal A} q^{\operatorname{st}(\pi)}$, whose coefficient of $x^j$ is the number of objects $\pi\in\mathcal A$
such that $\operatorname{st}(\pi)=j$. 
For a classical pattern $\tau$ and a quadrant marked mesh pattern
$\operatorname{MMP}(a,b,c,d)$, define
\begin{align*}
Q_{n,\tau}^{(a,b,c,d)}(q)
=
\sum_{\pi\in\operatorname{Av}_n(\tau)}q^{\operatorname{mmp}^{(a,b,c,d)}(\pi)}
\end{align*}
and
\begin{align*}
Q_{\tau}^{(a,b,c,d)}(t,q)
=
1+\sum_{n\geq1}t^nQ_{n,\tau}^{(a,b,c,d)}(q).
\end{align*}
Thus, the variable $t$ records the length of the permutation, whereas
$q$ records the number of matches of the specified quadrant marked mesh
pattern. 
In particular, we have
\begin{align*}
  s_n(\tau,P_{\ell,x})
  =[q^0]Q_{\tau,n}^{(\ell-x,x-1,0,0)}(q).
\end{align*}

Qiu and Remmel \cite[Conjecture~1]{QiuRemmel2018} conjectured that, for every
$k\geq1$,
\begin{align*}
Q_{132}^{(0,k,\emptyset,0)}(t,q)
=
Q_{132}^{(1,k-1,\emptyset,0)}(t,q).
\end{align*}
In Section~\ref{S4}, we correct an error in \cite{QiuRemmel2018} and then
prove that this conjecture holds.
Together with Qiu and Remmel's \cite[Corollary~1]{QiuRemmel2018}, this
immediately implies that $(231,P_{\ell,1})$ and $(231,P_{\ell,2})$ are
Wilf-equivalent. In Theorem~\ref{thm_231_method_1}, we provide an alternative
bijective proof of this equivalence.

This paper is organized as follows.
In Section~\ref{S2}, we establish the Wilf equivalences using bijective methods.
In Section~\ref{S3}, we use exact enumeration, together with an injection and a Diophantine equation, to prove that no further Wilf equivalences exist.
In Section~\ref{S4}, we extend the avoidance problem to the corresponding distribution problem and resolve a conjecture of Qiu and Remmel \cite{QiuRemmel2018}.

\section{Bijective proofs of Wilf equivalences}\label{S2}

In this section, we establish the Wilf equivalences appearing in our
classification. 
We begin with the simplest case. In \cite{KitaevRemmelTiefenbruckI}, 
Kitaev, Remmel, and Tiefenbruck used the complement and inverse operations to prove the following theorem.
\begin{thm}\label{231,Pl1}
For $\ell\geq 3$, $(231,P_{\ell,1})$ and $(213,P_{\ell,\ell})$ are Wilf-equivalent.
\end{thm}

Next, we turn our attention to the study of
$(132,P_{\ell,\ell})$ and $(132,P_{\ell,\ell-1})$. 
The following lemma gives a useful characterization of $132$-avoiding permutations in terms of 
$Q_{\mathrm{II}}(i)$.
We note that this lemma is the inverse-permutation version of
a classical result concerning inversion tables; see
\cite[Lemma 10]{ClaessonJelinekSteingrimsson2012}.
For completeness and consistency of notation, we restate their result
in our notation and include a proof.

\begin{lem}\label{lem:132-quadrant}
For $n\geq 3$, $\pi\in\text{Av}_n(132)$ if and only if $Q^\pi_{\mathrm{II}}(1)\ge Q^\pi_{\mathrm{II}}(2)\ge\cdots\ge Q^\pi_{\mathrm{II}}(n)$.
\end{lem}

\begin{proof}
If $\pi\in\operatorname{Av}_n(132)$ and
$Q^\pi_{\mathrm{II}}(i)<Q^\pi_{\mathrm{II}}(i+1)$, then $i+1$ must lie
to the right of $i$. Hence, all dots in the second quadrant of $i$ also lie
in the second quadrant of $i+1$. Since
$Q^\pi_{\mathrm{II}}(i+1)>Q^\pi_{\mathrm{II}}(i)$, there must be some
$k>i+1$ lying between $i$ and $i+1$. But then $i,k,i+1$ forms an occurrence
of $132$, a contradiction.

Suppose that $Q^\pi_{\mathrm{II}}(1)\geq Q^\pi_{\mathrm{II}}(2)\geq\cdots\geq
Q^\pi_{\mathrm{II}}(n)$, but $\pi\notin\operatorname{Av}_n(132)$. Choose an occurrence
$i_1i_3i_2$ of $132$ in $\pi$ such that $i_2-i_1$ is minimal among all
occurrences of $132$. If every dot in the second quadrant of $i_1$ also
lies in the second quadrant of $i_2$, then
$Q^\pi_{\mathrm{II}}(i_2)>Q^\pi_{\mathrm{II}}(i_1)$, a contradiction.
Therefore, there exists a dot $i_1'$ in the second quadrant of $i_1$ such
that $i_1'<i_2$. Then $i_1'i_3i_2$ is also an occurrence of $132$, but
$i_2-i_1'<i_2-i_1$,
contradicting the minimality of $i_2-i_1$.
\end{proof}


\begin{thm}\label{thm_231_method_1}
For $\ell\geq 3$, the pairs $(231,P_{\ell,1})$ and $(231,P_{\ell,2})$
are Wilf-equivalent.
\end{thm}

\begin{proof}
It suffices to construct a bijection $\varphi:
\operatorname{Av}_n(132,P_{\ell,\ell})
\longrightarrow
\operatorname{Av}_n(132,P_{\ell,\ell-1})$.
It is easy to see that, for every permutation in $S_n$ and every $i\in[n]$, $Q_{\mathrm{I}}(i)+Q_{\mathrm{II}}(i)=n-i$.
Moreover, a permutation $\eta$ contains $P_{\ell,x}$ if and only if there
exists an entry $i$ such that $Q_{\mathrm{II}}^\eta(i)\geq x-1$ and
$Q_{\mathrm{I}}^\eta(i)\geq \ell-x$.

We now define $\varphi$. Let $\pi\in\operatorname{Av}_n(132,P_{\ell,\ell})$.
Since $\pi$ avoids $132$, Lemma~\ref{lem:132-quadrant} implies that
$Q_{\mathrm{II}}^\pi(1)
\geq Q_{\mathrm{II}}^\pi(2)
\geq\cdots\geq Q_{\mathrm{II}}^\pi(n)$.
Moreover, since $\pi$ avoids $P_{\ell,\ell}$, for every $i\in[n]$,
$Q_{\mathrm{II}}^\pi(i)\leq\ell-2$.
If $Q_{\mathrm{II}}^\pi(1)\leq\ell-3$, then $Q_{\mathrm{II}}^\pi(i)\leq\ell-3$,
we set $\varphi(\pi)=\pi$.
In this case, $\pi$ automatically avoids $P_{\ell,\ell-1}$.

Suppose now that $Q_{\mathrm{II}}^\pi(1)=\ell-2$.
Let $k$ be the largest integer such that
$Q_{\mathrm{II}}^\pi(1)
=
Q_{\mathrm{II}}^\pi(2)
=
\cdots
=
Q_{\mathrm{II}}^\pi(k)
=
\ell-2$.
We can prove that $1,2,\ldots,k$ form a consecutive increasing block in
$\pi$.
First, for every $1\leq i<k$, the entry $i+1$ must lie to the right of
$i$. Otherwise $Q_{\mathrm{II}}^\pi(i)
\geq Q_{\mathrm{II}}^\pi(i+1)+1$, contrary to their equality. 
Therefore, $\operatorname{Pos}_{\pi}(1)
<
\operatorname{Pos}_{\pi}(2)
<
\cdots
<
\operatorname{Pos}_{\pi}(k)$.
Next, no entry can lie between $i$ and $i+1$ for
$1\leq i<k$. Otherwise, suppose that an entry $j$ satisfies
$
\operatorname{Pos}_{\pi}(i)
<
\operatorname{Pos}_{\pi}(j)
<
\operatorname{Pos}_{\pi}(i+1)$.
Since the entries $1,2,\ldots,i$ all occur at or before $i$, we must
have $j>i+1$. Then $Q_{\mathrm{II}}^\pi(i+1)
>
Q_{\mathrm{II}}^\pi(i)$,
again a contradiction. 
Thus, $1,2,\ldots,k$ form a consecutive block.

Since $Q_{\mathrm{II}}^\pi(1)=\operatorname{Pos}_{\pi}(1)-1=\ell-2$,
the permutation $\pi$ has a unique decomposition $\pi=A\,1\,2\cdots k\,B$,
where $|A|=\ell-2$. Define $\varphi(\pi)=A\,B\,k(k-1)\cdots2\,1$,
and set $\sigma=\varphi(\pi)$. The permutations $\pi$ and $\sigma$ are
illustrated in Figure~\ref{fig:phi}.

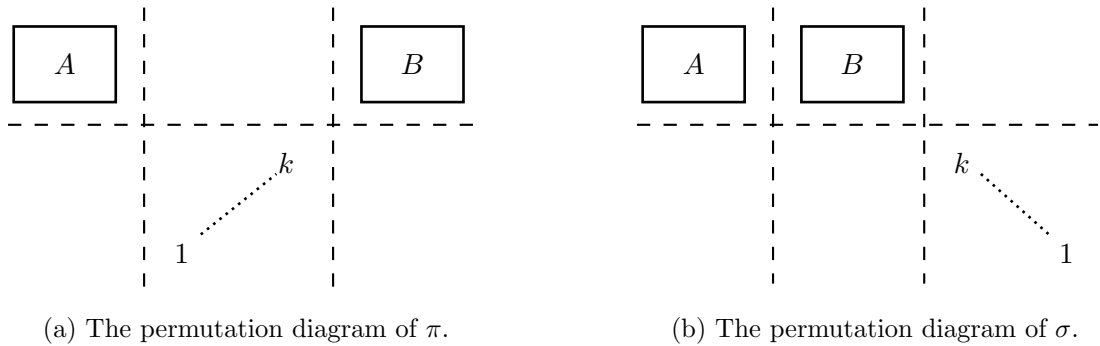
\begin{figure}[htbp]
\centering

\begin{subfigure}[t]{0.48\textwidth}
\centering
\begin{tikzpicture}[
    scale=1.0,
    guide/.style={
        dashed,
        dash pattern=on 5pt off 5pt,
        line width=0.8pt
    },
    dotguide/.style={
        dotted,
        line width=1pt
    },
    box/.style={
        draw,
        line width=1pt,
        minimum width=1.35cm,
        minimum height=1.0cm,
        inner sep=0pt
    }
]
\path[use as bounding box] (-1.3,-2.3) rectangle (5.0,1.55);

\draw[guide] (-1.3,0) -- (5.0,0);
\draw[guide] (0.5,1.55) -- (0.5,-2.3);
\draw[guide] (3,1.55) -- (3,-2.3);

\node[box] at (-0.55,0.8) {$A$};
\node[box] at (4.05,0.8) {$B$};

\node at (1.0,-1.7) {$1$};
\node at (2.38,-0.51) {$k$};
\draw[dotguide] (1.25,-1.45) -- (2.25,-0.65);
\end{tikzpicture}

\caption{The permutation diagram of $\pi$.}
\label{fig:pi}
\end{subfigure}
\hfill
\begin{subfigure}[t]{0.48\textwidth}
\centering
\begin{tikzpicture}[
    scale=1.0,
    guide/.style={
        dashed,
        dash pattern=on 5pt off 5pt,
        line width=0.8pt
    },
    dotguide/.style={
        dotted,
        line width=1pt
    },
    box/.style={
        draw,
        line width=1pt,
        minimum width=1.35cm,
        minimum height=1.0cm,
        inner sep=0pt
    }
]
\path[use as bounding box] (-1.3,-2.3) rectangle (5.0,1.55);

\draw[guide] (-1.3,0) -- (4.8,0);
\draw[guide] (0.5,1.55) -- (0.5,-2.1);
\draw[guide] (2.5,1.55) -- (2.5,-2.1);

\node[box] at (-0.55,0.8) {$A$};
\node[box] at (1.55,0.8) {$B$};

\node at (3.0,-0.51) {$k$};
\node at (4.38,-1.7) {$1$};
\draw[dotguide] (3.25,-0.65) -- (4.15,-1.45);
\end{tikzpicture}

\caption{The permutation diagram of $\sigma$.}
\label{fig:sigma}
\end{subfigure}

\caption{The map $\varphi$ from $\pi$ to $\sigma$.}
\label{fig:phi}
\end{figure}

We first show that $\sigma$ avoids $132$. For $1\leq i\leq k$, every
entry larger than $i$ lies to the left of $i$ in $\sigma$. Hence for $1\leq i\leq k$,
$Q_{\mathrm{II}}^\sigma(i)=n-i$,
and therefore
$Q_{\mathrm{II}}^\sigma(1)>
Q_{\mathrm{II}}^\sigma(2)>
\cdots>
Q_{\mathrm{II}}^\sigma(k)$.
For $j>k$, the relative order of all entries larger than $j$ is
unchanged by $\varphi$. Thus, for $j>k$,
$Q_{\mathrm{II}}^\sigma(j)
=
Q_{\mathrm{II}}^\pi(j)$.
By Lemma~\ref{lem:132-quadrant}, $\sigma$ avoids $132$.

It remains to show that $\sigma$ avoids $P_{\ell,\ell-1}$. 
For $1\leq i\leq k$, every entry larger than $i$ lies to its left, so
$Q_{\mathrm{I}}^\sigma(i)=0$.
On the other hand, for $i>k$, the maximality of $k$ gives
$Q_{\mathrm{II}}^\sigma(i)=Q_{\mathrm{II}}^\pi(i)\leq\ell-3$.
Thus, there is no entry $i$ for which both
$Q_{\mathrm{II}}^\sigma(i)\geq\ell-2$ and
$Q_{\mathrm{I}}^\sigma(i)\geq1$ hold. 
Hence, $\sigma$ avoids $P_{\ell,\ell-1}$.
The inverse map can be defined analogously, showing that $\varphi$ is indeed a bijection.
This completes the proof.
\end{proof}

Next, we study the equivalence between $(321,P_{\ell,x})$ and $(231,P_{\ell,x})$.
The Wilf equivalence underlying the following theorem also follows from
the results of Qiu and Remmel \cite{QiuRemmel2018},
expressed in the language of quadrant marked mesh patterns. 
Here we give a different bijective realization by showing that the maps $\Phi$ and
$\Psi$ defined below restrict to the corresponding avoidance classes.
For any $\pi\in S_n$, we define $\varphi^{\langle i\rangle}(\pi)$ as the permutation obtained
by partitioning the plane into quadrants with origin at $i$,
reordering the points in the second quadrant in decreasing order, and then inserting them
back into their original positions; and $\psi^{\langle i\rangle}$
as the analogous operation where the points in the second quadrant are reordered in increasing order. 
If we use $W_i(\pi)$ to denote the subword in $Q_{\mathrm{II}}(i)$, 
then $\varphi^{\langle i\rangle}(\pi)$ is obtained by rearranging the entries of $W_i(\pi)$ in decreasing order, and $\psi^{\langle i\rangle}(\pi)$ is obtained by rearranging the entries of $W_i(\pi)$ in increasing order.
We set
\begin{align*}
  \Phi=\varphi^{\langle n\rangle}\circ\cdots\circ\varphi^{\langle 1 \rangle},\quad\quad\Psi=\psi^{\langle n\rangle}\circ\cdots\circ\psi^{\langle 1\rangle}.
\end{align*}
We first show that $\Phi:\text{Av}_n(321)\to \text{Av}_n(231)$ and 
$\Psi:\text{Av}_n(231)\to\text{Av}_n(321)$.

\begin{lem}\label{lem:images-of-Phi-Psi}
For $n\geq 1$, the maps $\Phi$ and $\Psi$ satisfy
$\Phi(S_n)\subseteq \operatorname{Av}_n(231)$ and $\Psi(S_n)\subseteq \operatorname{Av}_n(321)$.
\end{lem}

\begin{proof}
It is easy to see that $\pi\in\operatorname{Av}_n(231)$ if and only if $W_i(\pi)$ is
decreasing for every $i\in[n]$.
Fix $\tau\in S_n$, and, for $1\leq a\leq n$, set
$\tau^{\langle a\rangle}
=
\bigl(
\varphi^{\langle a\rangle}\circ
\varphi^{\langle a-1\rangle}\circ\cdots\circ
\varphi^{\langle 1\rangle}
\bigr)(\tau)$.
We prove by induction on $a$ that
$W_i\bigl(\tau^{\langle a\rangle}\bigr)$ is decreasing for every $i\leq a$.
For $a=1$, this follows immediately from the definition of
$\varphi^{\langle1\rangle}$.

Suppose that the assertion holds after the first $a-1$ operations.
Applying $\varphi^{\langle a\rangle}$ makes
$W_a(\tau^{\langle a\rangle})$ decreasing by definition. It remains
to verify that this operation does not destroy the decreasing property
of $W_i$ for any $i<a$. Fix such an $i$, and consider the relative
positions of $i$ and $a$ in $\tau^{\langle a-1\rangle}$.

If $a$ lies to the left of $i$, then $W_a(\tau^{\langle a-1\rangle})$ is a subword of the
decreasing word $W_i(\tau^{\langle a-1\rangle})$, and hence it is
already decreasing. Therefore, $\varphi^{\langle a\rangle}$ does not
change $W_i$.

If $a$ lies to the right of $i$, write $W_i\bigl(\tau^{\langle a-1\rangle}\bigr)=UV$,
where $U$ consists of the entries larger than $a$, and $V$ consists
of the entries whose values lie strictly between $i$ and $a$.
Since $W_i(\tau^{\langle a-1\rangle})$ is decreasing, this
decomposition is well defined, and both $U$ and $V$ are decreasing.
The operation $\varphi^{\langle a\rangle}$ leaves all entries of $V$
fixed and rearranges only entries larger than $a$. After the operation,
the entries occupying the positions of $U$ form a decreasing word,
say $U'$. Moreover, each entry of $U'$ is larger than every entry of $V$. Hence
$W_i\bigl(\tau^{\langle a\rangle}\bigr)=U'V$
is still decreasing.

The induction is complete. Taking $a=n$, we conclude that
$\Phi(\tau)\in\operatorname{Av}_n(231)$.
The proof for $\Psi$ is analogous and is therefore omitted.
\end{proof}

Next, we obtain the following bijection. We provide two proofs that the map is indeed a bijection. The first is direct, while the second reveals its connection with the Simion--Schmidt bijection.

\begin{thm}\label{thm:Phi-Psi-bijection}
For $n\geq 1$, $\Phi$ and $\Psi$ are bijections between $\text{Av}_n(321)$ and $\text{Av}_n(231)$. In particular,
$\Phi\circ\Psi=\text{id}_{\text{Av}_n(231)}$, $\Psi\circ\Phi=\text{id}_{\text{Av}_n(321)}$.
\end{thm}
\begin{proof}
By Lemma~\ref{lem:images-of-Phi-Psi}, the two maps have the indicated codomains.
Fix $\pi\in\operatorname{Av}_n(321)$ and set $\sigma=\Phi(\pi)$.
For $0\leq a\leq n$, define $\sigma^{\langle0\rangle}=\sigma$
and $\sigma^{\langle a\rangle}
=
\bigl(
\psi^{\langle a\rangle}\circ\cdots\circ
\psi^{\langle1\rangle}
\bigr)(\sigma)$.
Thus, $\sigma^{\langle n\rangle}
=
\Psi(\sigma)
=
\Psi(\Phi(\pi))$.
Similar to the proof of Lemma~\ref{lem:images-of-Phi-Psi}, we have the following observation: in $\sigma^{\langle a-1\rangle}$, for every $i<a$, the entries in the second quadrant of $i$ are increasing. Applying $\psi^{\langle j\rangle}$ with $j>i$ does not destroy this increasing order; it only rearranges some larger entries into increasing order again.

We prove by induction on $a$ that, for $0\leq a\leq n-1$,
\begin{align}\label{eq:positions-restored}
\operatorname{Pos}_{\sigma^{\langle a\rangle}}(j)
=
\operatorname{Pos}_{\pi}(j)
\qquad
\text{for every }j\in[a+1].
\end{align}

For $a=0$, the claim is immediate.
Now let $1\leq a\leq n-1$ and assume that \eqref{eq:positions-restored} holds with $a-1$ in place of $a$.
Set $\rho=\sigma^{\langle a-1\rangle}$.
Thus, for $j\in[a]$, we have
$\operatorname{Pos}_{\rho}(j)
=
\operatorname{Pos}_{\pi}(j)$.
Moreover, by the construction of $\Psi$, the word $W_i(\rho)$ is
increasing for every $i<a$.

Since $\psi^{\langle a\rangle}$ moves only entries larger than $a$,
it does not change the positions of $1,2,\ldots,a$. It therefore
remains only to prove that $\operatorname{Pos}_{\sigma^{\langle a\rangle}}(a+1)
=
\operatorname{Pos}_{\pi}(a+1)$.

We first prove that
\begin{align}\label{eq:same-side}
\operatorname{Pos}_{\rho}(a+1)<\operatorname{Pos}_{\rho}(a)
\quad\Longleftrightarrow\quad
\operatorname{Pos}_{\pi}(a+1)<\operatorname{Pos}_{\pi}(a).
\end{align}

Suppose first that $a+1$ lies to the left of $a$ in $\rho$, but to the
right of $a$ in $\pi$. If some $i<a$ lies to the right of $a+1$ in
$\pi$, then, since the positions of $a$ and $i$ have already been
restored, the entries $a+1,a,i$ occur in this order in $\rho$. Hence
$W_i(\rho)$ contains $a+1$ before $a$, which contradicts
the fact that $W_i(\rho)$ is increasing.
If no such $i$ exists, then every entry smaller than $a+1$ lies to the
left of $a+1$ in $\pi$. $a+1$ is never moved past any smaller entry 
and therefore remains to the right of $a$, again a contradiction.

Conversely, suppose that $a+1$ lies to the left of $a$ in $\pi$. 
Since $\pi$ avoids $321$, no entry $i<a$ can lie to the right of $a$. 
Thus $a$ is the rightmost entry among $1,2,\ldots,a$. 
By the definition of $\rho$, we know that $a+1$ lies to the left of $a$ in $\rho$, which proves \eqref{eq:same-side}.
We now distinguish two cases.

\noindent \textbf{Case 1:} $a+1$ lies to the left of $a$ in $\rho$.

By \eqref{eq:same-side}, $a+1$ also lies to the left of $a$ in $\pi$.
Since the positions of $1,2,\ldots,a$ agree in $\rho$ and $\pi$, the
positions occupied by entries larger than $a$ to the left of $a$ are
the same in the two permutations. 
Because $\pi$ avoids $321$, the word $W_a(\pi)$ is increasing. Since
$a+1$ is its smallest possible entry, $a+1$ occupies the leftmost
position in $W_a(\pi)$. The operation $\psi^{\langle a\rangle}$
arranges $W_a(\rho)$ increasingly, and hence also places $a+1$ in this
same leftmost position. Therefore,
$\operatorname{Pos}_{\sigma^{\langle a\rangle}}(a+1)
=
\operatorname{Pos}_{\pi}(a+1)$.

\noindent \textbf{Case 2:} $a+1$ lies to the right of $a$ in $\rho$.

By \eqref{eq:same-side}, $a+1$ also lies to the right of $a$ in $\pi$.
If no entry smaller than $a+1$ lies to the right of $a+1$ in $\pi$, then
$a+1$ is never moved. Hence
$\operatorname{Pos}_{\rho}(a+1)
=
\operatorname{Pos}_{\pi}(a+1)$.
Suppose instead that there is an entry smaller than $a+1$ to its right,
and let $i$ be the rightmost entry among $1,2,\ldots,a$. Since $a$
lies to the left of $a+1$, we have $i<a$, and $i$ lies to the right of
$a+1$. Then $a+1$ remains to the left of $i$ throughout all the operations.
We know that $a+1$ is the leftmost entry larger than $a$ lying to the
right of $a$ in $\pi$. Indeed, if some entry $y>a$ occurred between
$a$ and $a+1$, then $y>a+1$. Then $y,a+1,i$ would
form an occurrence of $321$ in $\pi$, a contradiction.
$a+1$ is the leftmost entry larger than $a$ lying to the
right of $a$ in $\rho$. Otherwise, there would be an entry
$y>a$ between $a$ and $a+1$ in $\rho$. Since $a+1$ remains to the left
of $i$, the word $W_i(\rho)$ would contain $y$ before $a+1$,
contradicting the fact that $W_i(\rho)$ is increasing.
Finally, the positions of $1,2,\ldots,a$ are the same in $\rho$ and
$\pi$. Hence the leftmost position to the right of $a$ occupied by an
entry larger than $a$ is the same in both permutations. 
So we obtain
$
\operatorname{Pos}_{\rho}(a+1)
=
\operatorname{Pos}_{\pi}(a+1).
$
Moreover, since $a+1$ lies to the right of $a$,
$\psi^{\langle a\rangle}$ does not move it. Thus
$
\operatorname{Pos}_{\sigma^{\langle a\rangle}}(a+1)
=
\operatorname{Pos}_{\pi}(a+1).
$

This completes the induction. Hence,
$\Psi\circ\Phi
=
\mathrm{id}_{\operatorname{Av}_n(321)}$.
Since both $\operatorname{Av}_n(321)$ and
$\operatorname{Av}_n(231)$ have cardinality equal to the $n$th
Catalan number, it follows that
$\Phi\circ\Psi
=
\mathrm{id}_{\operatorname{Av}_n(231)}$.
\end{proof}

Next, we describe the relationship between our bijection and the
Simion--Schmidt bijection, providing a second proof that our map is indeed
a bijection. For $\pi=\pi_1\pi_2\cdots\pi_n\in S_n$, let
$\mathfrak{lmin}(\pi)
=
\left\{
(i,\pi_i):
\pi_i<\pi_j\text{ for every }j<i
\right\}$
and
$\mathfrak{rmin}(\pi)
=
\left\{
(i,\pi_i):
\pi_i<\pi_j\text{ for every }j>i
\right\}$
denote, respectively, the sets of {\em left-to-right minima} and
{\em right-to-left minima} of $\pi$. The following characterization of the Simion--Schmidt bijection is due
to Claesson and Kitaev \cite{ClaessonKitaev2008}.

\begin{lem}[{\cite[Lemma 6]{ClaessonKitaev2008}}]
\label{lem:Simion--Schmidt-lmin}
Let $\pi\in\operatorname{Av}_n(123)$ and
$\sigma\in\operatorname{Av}_n(132)$. Then the following statements
are equivalent:
\begin{enumerate}
    \item $\operatorname{SS}(\pi)=\sigma$, where
    $\operatorname{SS}$ denotes the Simion--Schmidt bijection;
    \item $\mathfrak{lmin}(\pi)=\mathfrak{lmin}(\sigma)$.
\end{enumerate}
\end{lem}
We can now prove the following theorem, thereby showing that $\Phi$ is indeed a bijection.

\begin{thm}\label{thm:Phi-Simion--Schmidt} We have
\begin{align*}
\operatorname{rev}\circ\ \Phi\circ\operatorname{rev}
=
\operatorname{SS},
\end{align*}
where $\operatorname{rev}(\pi)=\pi^r$.
\end{thm}

\begin{proof}
Fix $\pi\in\operatorname{Av}_n(123)$, and set
$\tau=\operatorname{rev}(\pi)$,
$\rho=\Phi(\tau)$, $\sigma=\operatorname{rev}(\rho)$.
Then
$\tau\in\operatorname{Av}_n(321)$, $\rho\in\operatorname{Av}_n(231)$,
$\sigma\in\operatorname{Av}_n(132)$.
By Lemma~\ref{lem:Simion--Schmidt-lmin}, it suffices to prove that
$
\mathfrak{rmin}(\tau)=\mathfrak{rmin}(\Phi(\tau))$.
Recall that
$
\Phi
=
\varphi^{\langle n\rangle}\circ
\varphi^{\langle n-1\rangle}\circ\cdots\circ
\varphi^{\langle1\rangle}$.
We show that each operation $\varphi^{\langle i\rangle}$ preserves the
right-to-left minima.

The operation $\varphi^{\langle i\rangle}$ rearranges in decreasing order the entries lying in the second quadrant relative to $i$. Every such entry is larger than $i$ and lies to the left of $i$. Hence, none of these entries is a right-to-left minimum, since the smaller entry $i$ lies to its right. Moreover, none of them becomes a right-to-left minimum after the rearrangement.

It remains to consider the entries not moved by
$\varphi^{\langle i\rangle}$. Let $h$ be such an entry. If $h$ lies to
the left of $i$, then necessarily $h<i$. Every moved entry
$j$ satisfies $j>i>h$.
Therefore, rearranging these entries does not affect whether $h$ has
a smaller entry to its right.

If $h=i$ or $h$ lies to the right of $i$, then all entries moved by
$\varphi^{\langle i\rangle}$ lie to the left of $h$. Consequently, the
entries lying to the right of $h$ are unchanged, so the status of $h$
as a right-to-left minimum is also unchanged.

By induction, we conclude that
$\mathfrak{lmin}(\pi)=\mathfrak{lmin}(\sigma)$.
So we have
$
\operatorname{rev}\circ\ \Phi\circ\operatorname{rev}
=
\operatorname{SS}$.
\end{proof}
Next, we use $\Phi$ to establish the equivalence between
 $(321,P_{\ell,x})$ and $(231,P_{\ell,x})$ for $2\leq x \leq \ell$.

\begin{thm}\label{thm:Phi-Psi-preserve-POP}
For $2\leq x\leq\ell$, the maps $\Phi$ and $\Psi$ restrict to
mutually inverse bijections
$\Phi:
\operatorname{Av}_n(321,P_{\ell,x})
\longrightarrow
\operatorname{Av}_n(231,P_{\ell,x})$
and
$
\Psi:
\operatorname{Av}_n(231,P_{\ell,x})
\longrightarrow
\operatorname{Av}_n(321,P_{\ell,x})$.
\end{thm}

\begin{proof}
Let $\pi\in\operatorname{Av}_n(321,P_{\ell,x})$ and
$\sigma\in\operatorname{Av}_n(231,P_{\ell,x})$.
For convenience, define
$
\pi^{\langle a\rangle}
=
\bigl(
\varphi^{\langle a\rangle}\circ\cdots\circ
\varphi^{\langle1\rangle}
\bigr)(\pi)$
and
$
\sigma^{\langle a\rangle}
=
\bigl(
\psi^{\langle a\rangle}\circ\cdots\circ
\psi^{\langle1\rangle}
\bigr)(\sigma)$.
Recall that an entry $v$ can occupy the $x$-th position of an occurrence
of $P_{\ell,x}$ if and only if
$Q_{\mathrm{II}}(v)\geq x-1$ and
$Q_{\mathrm I}(v)\geq\ell-x$.
We first prove that $\Phi$ preserves avoidance of $P_{\ell,x}$. We
show inductively that if an occurrence of $P_{\ell,x}$ appears in
$\pi^{\langle a\rangle}$, then none of $1,2,\ldots,a+1$
can occupy its $x$-th position.

First consider the entry $1$. It is easy to check that rearranging entries does not change
the numbers of entries in the first and second quadrants of $1$.
Therefore,
$Q_{\mathrm I}^{\pi^{\langle a\rangle}}(1)
=
Q_{\mathrm I}^{\pi}(1)$ and
$Q_{\mathrm{II}}^{\pi^{\langle a\rangle}}(1)
=
Q_{\mathrm{II}}^{\pi}(1)$.
Since $\pi$ avoids $P_{\ell,x}$, the entry $1$ cannot occupy the
$x$-th position of an occurrence of $P_{\ell,x}$ in any
$\pi^{\langle a\rangle}$.

Now suppose inductively that, in $\pi^{\langle a-1\rangle}$, none of
$1,\ldots,a$ can occupy the $x$-th position of an occurrence. The
permutations $\pi^{\langle a-1\rangle}$ and
$\pi^{\langle a\rangle}$ differ only by
$\varphi^{\langle a\rangle}$, which rearranges only entries larger than
$a$. Hence, for every $j\leq a$,
$Q_{\mathrm I}^{\pi^{\langle a\rangle}}(j)
=
Q_{\mathrm I}^{\pi^{\langle a-1\rangle}}(j)$
and
$Q_{\mathrm{II}}^{\pi^{\langle a\rangle}}(j)
=
Q_{\mathrm{II}}^{\pi^{\langle a-1\rangle}}(j)$.
Thus none of $1,\ldots,a$ can occupy the $x$-th position in
$\pi^{\langle a\rangle}$. It remains to consider $a+1$.
We distinguish two cases.

\noindent \textbf{Case 1:} $a+1$ has not been acted on by any of
$\varphi^{\langle1\rangle},\ldots,\varphi^{\langle a\rangle}$.

In this case, whenever $\varphi^{\langle i\rangle}$ is applied, the entry
$i$ lies to the left of $a+1$. 
Thus all the rearrangements take place
to the left of $a+1$, and hence
$Q_{\mathrm I}^{\pi^{\langle a\rangle}}(a+1)
=
Q_{\mathrm I}^{\pi}(a+1)$
and
$Q_{\mathrm{II}}^{\pi^{\langle a\rangle}}(a+1)
=
Q_{\mathrm{II}}^{\pi}(a+1)$.
So $a+1$ cannot occupy the $x$-th position in $\pi^{\langle a\rangle}$.

\noindent \textbf{Case 2:} $a+1$ has been acted on by some
$\varphi^{\langle i\rangle}$.

Let $\varphi^{\langle i\rangle}$ be the last operation acting on $a+1$.
Suppose, for contradiction, that $a+1$ occupies the $x$-th position
of an occurrence of $P_{\ell,x}$. Then
$Q_{\mathrm{II}}^{\pi^{\langle a\rangle}}(a+1)\geq x-1$ and
$Q_{\mathrm I}^{\pi^{\langle a\rangle}}(a+1)\geq\ell-x$.
So we have 
$
Q_{\mathrm{II}}^{\pi^{\langle a\rangle}}(i)
\geq
Q_{\mathrm{II}}^{\pi^{\langle a\rangle}}(a+1)+1
\geq x$.
Since $i\leq a$, the induction hypothesis implies that $i$ cannot
occupy the $x$-th position of an occurrence. Therefore,
$Q_{\mathrm I}^{\pi^{\langle a\rangle}}(i)<\ell-x$.
If $x=\ell$, the last inequality is already impossible. Otherwise,
since
$Q_{\mathrm I}^{\pi^{\langle a\rangle}}(a+1)\geq\ell-x$ and
$Q_{\mathrm I}^{\pi^{\langle a\rangle}}(i)<\ell-x$,
there exists an entry $j>a+1$ such that
$\operatorname{Pos}_{\pi^{\langle a\rangle}}(a+1)
<
\operatorname{Pos}_{\pi^{\langle a\rangle}}(j)
<
\operatorname{Pos}_{\pi^{\langle a\rangle}}(i)$.
This contradicts the fact that the
second quadrant subword of $i$ is decreasing. 
Hence $a+1$ cannot
occupy the $x$-th position of an occurrence of $P_{\ell,x}$.
This completes the induction. Therefore, $\Phi(\pi)\in\operatorname{Av}_n(231,P_{\ell,x})$.

We next prove that $\Psi$ preserves avoidance of $P_{\ell,x}$. 
We proceed with the same line of reasoning to prove it.

The position of $1$ does not change, and neither do its
first and second quadrant counts. 
Now suppose the assertion holds
for $\sigma^{\langle a-1\rangle}$. Since
$\psi^{\langle a\rangle}$ rearranges only entries larger than $a$,
it remains only to consider $a+1$.
Again, there are two cases.

\noindent \textbf{Case 1:} $a+1$ has not been acted on by any
$\psi^{\langle i\rangle}$.

In this case, the numbers of entries in the first and second quadrants
of $a+1$ are unchanged. Since $\sigma$ avoids $P_{\ell,x}$, the entry
$a+1$ cannot occupy the $x$-th position of an occurrence in
$\sigma^{\langle a\rangle}$.

\noindent \textbf{Case 2:} $a+1$ has been acted on by some
$\psi^{\langle i\rangle}$.

Let $\psi^{\langle i\rangle}$ be the last operation acting on $a+1$.
This operation arranges the second-quadrant subword of $i$ in
increasing order. Consequently, there is no entry larger than $a+1$
to the left of $a+1$. Hence
$Q_{\mathrm{II}}^{\sigma^{\langle a\rangle}}(a+1)=0$.
Since $x\geq2$, we have
$Q_{\mathrm{II}}^{\sigma^{\langle a\rangle}}(a+1)
=0<x-1$.
Therefore $a+1$ cannot occupy the $x$-th position of an occurrence of
$P_{\ell,x}$.
The induction is complete, and hence
$\Psi(\sigma)\in\operatorname{Av}_n(321,P_{\ell,x})$.
\end{proof}

The bijection in Theorem~\ref{thm:Phi-Psi-preserve-POP} implies that, for every $x$ with $2\leq x\leq\ell$, the pairs $(321,P_{\ell,x})$ and $(231,P_{\ell,x})$ are Wilf-equivalent.

\section{Non-Wilf-equivalence of certain pairs of patterns}\label{S3}

We begin by considering the special case of $(321,P_{\ell,1})$. To determine
its avoidance sequence and distinguish this class from all the remaining
pairs, we need a refinement of the Catalan enumeration of $123$-avoiding
permutations according to a prescribed decreasing suffix. This refinement is
given by the Catalan triangle and is recorded in the following lemma.

\begin{lem}\label{321_de}
For $n\geq k\geq1$, the number of permutations in $\operatorname{Av}_n(123)$ with a decreasing
suffix of length $k$ is the Catalan triangle number
\begin{equation*}
C(n,n-k)=\frac{k+1}{n+1}\binom{2n-k}{n}.    
\end{equation*}
\end{lem}

\begin{proof}
Let $a_{n,k}$ be the number of permutations in $\operatorname{Av}_n(123)$
with a decreasing suffix of length $n-k$. Since $\pi$ avoids $123$, the
entries to the right of $1$ must be decreasing. If
$\operatorname{Pos}_\pi(1)=i$, then deleting $1$ yields
$a_{n-1,i-1}$ possibilities. Therefore,
$a_{n,k}=\sum_{i=0}^k a_{n-1,i}$, with $a_{n,0}=1$.
This is precisely the recurrence defining the Catalan triangle numbers.
Hence, $a_{n,k}=C(n,k)$.
\end{proof}

\begin{thm}\label{thm-123-Pll}
For $n\ge\ell$,
\begin{equation*}
    s_n(123,P_{\ell,\ell})=\frac{2n-2\ell+3}{n+1}\binom{2\ell-2}n.
\end{equation*}
\end{thm}

\begin{proof}
Let $n\geq\ell$, let $a(n)=s_n(123,P_{\ell,\ell})$, and let $h_i(n)$
denote the number of permutations in
$\operatorname{Av}_n(123,P_{\ell,\ell})$ with a decreasing suffix of length
$n-\ell+i+1$. For any
$\pi\in\operatorname{Av}_n(123,P_{\ell,\ell})$, we have
$\operatorname{Pos}_\pi(1)=j$ for some $1\leq j\leq\ell-1$.
Since the entries to the right of $1$ must be decreasing, deleting $1$
yields $h_{\ell-j}(n-1)$ possibilities. Therefore
\begin{align*}
  a(n)=h_0(n)=\sum_{i=1}^{\ell-1}h_{\ell-i}(n-1).
\end{align*}
For any $\pi\in\operatorname{Av}_n(123,P_{\ell,\ell})$ with a decreasing
suffix of length $n-\ell+k+1$, we have
$1\leq \operatorname{Pos}_\pi(1)\leq \ell-k-1$. So
\begin{align*}
h_k(n)=\sum_{i=1}^{\ell-k-1}h_{\ell-i}(n-1).
\end{align*}
In particular, $h_{\ell-1}(n)=0$ for $n\geq\ell$. By
Lemma~\ref{321_de},
\begin{align*}
h_k(\ell-1)=\frac{k+1}{\ell}\binom{2\ell-k-2}{\ell-1}.
\end{align*}
We now prove by induction that
\begin{align*}
h_k(\ell-1+t)
=
\dfrac{2t+k+1}{\ell+t}
\binom{2\ell-k-2}{\ell+t-1}.
\end{align*}
For $k=\ell-1$, the formula is immediate.
Suppose that the formula holds for $k+1$, then
\begin{align}
h_k(\ell-1+t)=&\sum_{i=1}^{\ell-k-2}h_{\ell-i}(\ell-2+t)+h_{k+1}(\ell-2+t)\notag\\
=&h_{k+1}(\ell-1+t)+h_{k+1}(\ell-2+t)\notag\\
=&\dfrac{2t+k+2}{\ell+t}\binom{2\ell-k-3}{\ell+t-1}+\dfrac{2t+k}{\ell+t-1}\binom{2\ell-k-3}{\ell+t-2}\notag\\
=&\dfrac{2t+k+2}{\ell+t}\binom{2\ell-k-3}{\ell+t-1}+\dfrac{2t+k}{\ell-t-k-1}\binom{2\ell-k-3}{\ell+t-1}\notag\\
=&\dfrac{2t+k+1}{\ell+t}\binom{2\ell-k-2}{\ell+t-1}\notag
\end{align}
So $a(n)=h_0(\ell-1+(n-\ell+1))=\frac{2n-2\ell+3}{n+1}\binom{2\ell-2}n$.
\end{proof}
\begin{cor}\label{cor-321}
The set $\{(321,P_{\ell,1}),(123,P_{\ell,\ell})\}$ is a
Wilf-equivalence class.
\end{cor}

\begin{proof}
By Theorem~\ref{thm-123-Pll}, for $n\geq 2\ell-1$, there are no
permutations of length $n$ avoiding both patterns in either of the two
pairs in question. However, for every other pair $(\tau,P_{\ell,x})$,
where $\tau$ is a pattern of length three, at least one of the permutations
$12\cdots n$ and $n(n-1)\cdots1$ avoids both $\tau$ and $P_{\ell,x}$.
Hence, the corresponding avoidance class is nonempty for arbitrarily large
$n$. Therefore, $(321,P_{\ell,1})$ and $(123,P_{\ell,\ell})$ are not
Wilf-equivalent to any other pair.
\end{proof}

Next, we show that, for $2\leq x\leq \ell$, the classes $(231,P_{\ell,x})$ are pairwise inequivalent. 
We prove this result by constructing an injection.
Define
$\mathcal C_n=\{(c_1,\dots,c_n)\in\mathbb Z^n:0\le c_i\le n-i\}$.
For $\pi\in S_n$, we can define $\xi(\pi)=(Q^\pi_{\mathrm{II}}(1),\dots,Q^\pi_{\mathrm{II}}(n))$.
Conversely, given any integer sequence $c=(c_1,\dots,c_n)\in\mathcal C_n$,
we can construct a permutation $\eta(c)\in S_n$. Start with $n$, and then insert
$n-1,n-2,\dots,1$. When inserting $i$, the entries already placed are precisely
$i+1,\dots,n$, all of which are larger than $i$. Since $0\le c_i\le n-i$,
there is a unique way to insert $i$ so that exactly $c_i$ of the already
placed entries lie to its left. Clearly $\xi\circ\eta=\text{id}$ and
$\eta\circ\xi=\text{id}$.
The following characterization is equivalent, by reverse-complement, 
to a characterization of inversion tables of $312$-avoiding permutations given by Bényi \cite{Benyi2014}. 
We restate the result and include a short proof for completeness.

\begin{lem}\label{lem_231_encoding}
For $n\geq 3$, $\pi\in\text{Av}_n(231)$ if and only if for $(c_1,\dots,c_n)=\xi(\pi)$,
$c_i\le c_{i+1}+1,i\in[n-1]$.
\end{lem}

\begin{proof}
Suppose first that, for some $i$, we have $c_i>c_{i+1}+1$. Then $\text{Pos}_\pi(i+1)<\text{Pos}_\pi(i)$
and there must be at least one element $j$ larger than $i$ located between $i$
and $i+1$. But $i+1,j,i$ form a $231$ pattern. Therefore, if $\pi$
avoids $231$, then necessarily $c_i\le c_{i+1}+1$.

Conversely, suppose that $\pi$ contains a $231$ pattern.
Choose an occurrence $i_2,i_3,i_1$ such that $i_1<i_2<i_3$, and
$\operatorname{Pos}_{\pi}(i_2)
<
\operatorname{Pos}_{\pi}(i_3)
<
\operatorname{Pos}_{\pi}(i_1)$,
and such that $i_2-i_1$ is as small as possible among all occurrences
of $231$ in $\pi$.
We first prove that $i_2=i_1+1$.
Suppose instead that $i_2-i_1>1$, and set $d=i_1+1$.
Then $i_1<d<i_2<i_3$.
Consider the position of $d$ relative to $i_3$.
If $d$ lies to the left of $i_3$, then
$d,i_3,i_1$
forms an occurrence of $231$. But
$d-i_1=1<i_2-i_1$,
contradicting the minimality of $i_2-i_1$.
If $d$ lies to the right of $i_3$, then
$i_2,i_3,d$
forms an occurrence of $231$. Moreover,
$i_2-d=i_2-i_1-1<i_2-i_1,$
again contradicting the minimality of $i_2-i_1$.
Therefore $i_2=i_1+1$. Writing $i=i_1$ and $j=i_3$, we have an
occurrence $i+1,j,i$ of $231$. 
Hence $c_i\geq c_{i+1}+2$.
Thus the inequalities cannot all hold.
\end{proof}
By Lemma~\ref{lem_231_encoding}, 231-avoiding permutations can be encoded by 
non-decreasing integer sequences $y_1,\ldots,y_n$ with $i\le y_i\le n$. 
It remains to determine how the avoidance of $P_{\ell,x}$ is reflected in this encoding. 
This gives the following characterization.

\begin{lem}\label{lem_231}
Let $\mathcal C_n(a,b)$ be the set of nondecreasing integer sequences
$(y_1,\dots,y_n)$ of length $n$ satisfying $i\leq y_i\leq n$ for every
$i\in[n]$, and such that there is no $i\in[n]$ for which
$i+a\leq y_i\leq n-b$. Then
$s_n(231,P_{\ell,x})=|\mathcal C_n(x-1,\ell-x)|$.\end{lem}

\begin{proof}
For any $i$ in permutation $\pi$, let $y_i=i+c_i$. The number of dots in
the second quadrant of $i$ is $y_i-i$ and the number of dots in
the first quadrant of $i$ is $n-y_i$. If $\pi$ avoids $P_{\ell,x}$, then
there is no $i\in[n]$ such that $y_i-i\ge x-1$ and $n-y_i\ge\ell-x$.
That is $x-1+i\le y_i\le n-\ell+x$. So we construct a bijection between
$\text{Av}_n(231,P_{\ell,x})$ and $\mathcal C_n(x-1,\ell-x)$.
\end{proof}
Next, we use an injection to show that the $(231,P_{\ell,x})$ are pairwise inequivalent 
when $x\ge 2$.
\begin{thm}\label{231_inj}
For any $n\ge\ell$, $s_n(231,P_{\ell,\ell})>s_n(231,P_{\ell,\ell-1})>\cdots>s_n(231,P_{\ell,2})$.
\end{thm}
\begin{proof}
By Lemma~\ref{lem_231}, it suffices to prove that, whenever $a+b=\ell-1$, and $1\le a\le \ell-2$,
we have
$|\mathcal C_n(a,b)|
<
|\mathcal C_n(a+1,b-1)|$.

For convenience, set $T=n-b+1$. Then $y=(y_1,\ldots,y_n)\in\mathcal C_n(a,b)$ 
implies that there is no $i\in[n]$ such that $i+a\le y_i\le T-1$.
Hence, for every $i$, $y_i\le i+a-1$ or $y_i\ge T$.
Similarly, if $z=(z_1,\ldots,z_n)\in\mathcal C_n(a+1,b-1)$,
then for every $i$, we have $z_i\le i+a$ or $z_i>T$.

Define $\lambda:\mathcal C_n(a,b)\longrightarrow
\mathcal C_n(a+1,b-1)$
by $\lambda(y)=z=(z_1,\ldots,z_n)$, where
\begin{align*}
z_i=
\begin{cases}
\min\{y_i,i+a\}, & y_i\le T,\\
y_i, & y_i>T.
\end{cases}
\end{align*}
We next show that $\lambda$ is injective. Define $\mu$ by
\begin{align*}
\mu(z)_i=
\begin{cases}
T, & z_i=i+a<T,\\
z_i, & \text{otherwise}.
\end{cases}
\end{align*}
It is easy to check that $\mu\circ\lambda=\operatorname{id}$.
It remains to show that $\lambda$ is not surjective. Define
$z^*=(z_1^*,\ldots,z_n^*)$, where
$z_i^*=\max\{i,a+1\}$. Since $a\geq1$, we have
\begin{align*}
z^*=(\underbrace{a+1,\ldots,a+1}_{a+1\text{ terms}},a+2,a+3,\ldots,n).
\end{align*}
Hence, $z^*\in\mathcal C_n(a+1,b-1)$.

Suppose, to the contrary, that there exists
$y^*=(y_1^*,\ldots,y_n^*)\in\mathcal C_n(a,b)$ such that $\lambda(y^*)=z^*$.
Now $\lambda(y^*_1)=z_1^*=a+1<T$.
Since $T=n-b+1=n-\ell+a+2$, and $n\ge\ell$, we have $T\ge a+2>a+1$.
Since \(y^*\in\mathcal C_n(a,b)\), the first coordinate satisfies
$y_1^*\le a$ or $y_1^*\ge T$.
The equality \(\lambda(y^*_1)=a+1\) therefore forces $y_1^*=T$.
On the other hand, since $a\ge1$, we have $z_2^*=a+1$.
Because $a+1<T$, the equality $\lambda(y^*_2)=a+1$
forces $y_2^*=a+1$.
Thus $y_1^*=T>a+1=y_2^*$,
contradicting the fact that $y^*$ is non-decreasing.
Therefore $z^*\notin\operatorname{Im}(\lambda)$, so
$\lambda$ is not surjective. Hence
$|\mathcal C_n(a,b)|
<
|\mathcal C_n(a+1,b-1)|$.
This completes the proof.
\end{proof}

\begin{cor}\label{231_neq}
For $2\le x\le\ell$, $(231,P_{\ell,x})$ are pairwise non-Wilf-equivalent.
\end{cor}

It is worth noting that, when $x=1$, the injective map $\lambda$ in the proof of Theorem~\ref{231_inj} is actually
a bijection, thus providing an alternative proof of
Theorem~\ref{thm_231_method_1}. Next, by evaluating the generating functions
at suitable specializations, we determine the equivalence relations among
$(213,P_{\ell,x})$.

\begin{lem}\label{lem:recurrence-213}
Let $a_{\ell,x}(n)=s_n(213,P_{\ell,x})$.
For $n\geq\ell$, we have
\begin{align}
a_{\ell,x}(n)
={}
\sum_{r=0}^{x-2}
C_r\,a_{\ell-r,x-r}(n-r-1)
+
\sum_{r=0}^{\ell-x-1}
C_r\,a_{\ell,x}(n-r-1),
\label{eq:recurrence-213}
\end{align}
where $C_r$ denotes the $r$th Catalan number, and an empty sum is
understood to be zero. Moreover, for $0\leq n\leq\ell-1$, we have
$a_{\ell,x}(n)=C_n$.
\end{lem}

\begin{proof}
Let $\pi\in\operatorname{Av}_n(213,P_{\ell,x})$
and suppose that $\operatorname{Pos}_{\pi}(1)=i$.
Since $\pi$ avoids $213$, we can write $\pi=A1B$, where every entry of
$A$ is larger than every entry of $B$, both $A$ and $B$ avoid $213$, and
$|A|=i-1$ and $|B|=n-i$.
In order to avoid $P_{\ell,x}$, we must have either
$1\leq i\leq x-1$ or $n-\ell+x+1\leq i\leq n$.
We enumerate the permutations in these two cases separately.

If $1\leq i\leq x-1$, we have $|A|=i-1\leq x-2<\ell$,
so no entry of $A$ can be the distinguished entry of an occurrence
of $P_{\ell,x}$. Thus $A$ is only required to avoid $213$, and there
are $C_{i-1}$ possible choices for $A$.
Every entry of $A$ precedes and is larger than every entry of $B$.
It follows that $\pi$ avoids $P_{\ell,x}$ if and only if $B$ avoids
$P_{\ell-i+1,x-i+1}$.
Therefore, for a fixed $i$ in this range, the number of possible
permutations is $C_{i-1}\,
a_{\ell-i+1,x-i+1}(n-i)$.

If $n-\ell+x+1\leq i\leq n$, we have $|B|=n-i\leq\ell-x-1$.
Thus neither $1$ nor an entry of $B$ can be the distinguished entry
of an occurrence of $P_{\ell,x}$.
Moreover, every entry of $B$, as well as the entry $1$, is smaller
than every entry of $A$. Hence these entries cannot be used together
with an entry of $A$ as the distinguished entry of an occurrence of
$P_{\ell,x}$. Consequently, any occurrence of $P_{\ell,x}$ must lie
entirely in $A$.
It follows that $A$ must avoid both $213$ and $P_{\ell,x}$, whereas
$B$ is only required to avoid $213$. Therefore, for a fixed $i$ in
this range, the number of possible permutations is
$a_{\ell,x}(i-1)\,C_{n-i}$.
By a straightforward simplification, we obtain \eqref{eq:recurrence-213}
\end{proof}

\begin{thm}\label{213_no}
For $1\le x\le\ell$, $(213,P_{\ell,x})$ are not equivalent to each other.
\end{thm}

\begin{proof}
    By Lemma~\ref{lem:recurrence-213}, we have
    \begin{align}\label{eq:b-formula}
        a_{\ell,x}(\ell)=\sum_{i=0}^{x-2}C_iC_{\ell-i-1}+\sum_{j=0}^{\ell-x-1}C_jC_{\ell-j-1}=C_\ell-C_{x-1}C_{\ell-x}.
    \end{align}
Let $b(x)=a_{\ell,x}(\ell)$. The sequence $b(x)$ is symmetric, and
\begin{align*}
b(x+1)-b(x)
=
C_{x-1}C_{\ell-x}-C_{\ell-1-x}C_x.
\end{align*}
By the log-convexity of the Catalan numbers
\cite{LiuWang2007}, we have
$b(x+1)>b(x)$ whenever $x<\ell/2$. Hence, $b(x)$ is unimodal.

It remains only to prove that $(213,P_{\ell,x})$ and
$(213,P_{\ell,\ell-x+1})$ are not Wilf-equivalent.
For this purpose, define $c(x)=a_{\ell,x}(\ell+1)$.
For $1\le x\le \ell-1$, \eqref{eq:recurrence-213} gives
\begin{align*}
c(x)
={}&
\sum_{r=0}^{x-2}
C_r\,a_{\ell-r,x-r}(\ell-r)
+
\sum_{r=0}^{\ell-x-1}
C_r\,a_{\ell,x}(\ell-r).
\end{align*}
By \eqref{eq:b-formula}, $a_{\ell-r,x-r}(\ell-r)
=
C_{\ell-r}
-
C_{x-r-1}C_{\ell-x}$.
In the second sum, the term corresponding to $r=0$ is
$a_{\ell,x}(\ell)
=
C_\ell-C_{x-1}C_{\ell-x}$,
whereas for $r\ge1$ we have
$a_{\ell,x}(\ell-r)=C_{\ell-r}$.
Therefore
\begin{align*}
c(x)
={}
\sum_{r=0}^{x-2}
C_r\bigl(C_{\ell-r}-C_{x-r-1}C_{\ell-x}\bigr)
+C_\ell-C_{x-1}C_{\ell-x}
+\sum_{r=1}^{\ell-x-1}C_rC_{\ell-r}.
\end{align*}
Since $\sum_{r=0}^{x-2}C_rC_{x-r-1}
=
C_x-C_{x-1}$,
this simplifies to
\begin{equation*}\label{eq:c-formula}
c(x)
=
C_\ell
+\sum_{r=0}^{x-2}C_rC_{\ell-r}
+\sum_{r=1}^{\ell-x-1}C_rC_{\ell-r}
-C_xC_{\ell-x}.
\end{equation*}
Now suppose that $2\le x<\frac{\ell+1}{2}$.
Hence we have
\begin{equation}\label{eq:c-difference}
c(x)-c(\ell+1-x)
=
C_{x-1}C_{\ell+1-x}
-
C_xC_{\ell-x}.
\end{equation}
Because $\ell+1-x>x$ and the ratios $C_m/C_{m-1}$ are strictly
increasing, \eqref{eq:c-difference} is strictly positive.

It remains only to distinguish the endpoint pair $x=1$ and $x=\ell$.
A direct calculation gives
$
 c(1)-c(\ell)
=
C_\ell-2C_{\ell-1}
=
\frac{2(\ell-2)}{\ell+1}C_{\ell-1}>0$,   
since $\ell\ge3$.
Therefore the pairs $(213,P_{\ell,x})$, for $1\le x\le\ell$, are
pairwise non-Wilf-equivalent.
\end{proof}
It remains only to prove that, for $1\leq x\leq\ell$ and $3\leq y\leq\ell$,
the pairs $(213,P_{\ell,x})$ and $(321,P_{\ell,y})$ are not Wilf-equivalent.
By the proof of Theorem~\ref{213_no}, we have
$s_\ell(213,P_{\ell,x})=C_\ell-C_{x-1}C_{\ell-x}$.
Moreover,
$s_\ell(321,P_{\ell,y})
=
C_\ell-\#\{\pi\in\operatorname{Av}_\ell(321):
\operatorname{Pos}_\pi(1)=y\}$.
The subsequence to the left of $1$ must be increasing. So there is a bijection
\begin{align*}
\left\{
\pi\in\operatorname{Av}_\ell(321):
\operatorname{Pos}_\pi(1)=y
\right\}
\longleftrightarrow
\left\{
\tau\in\operatorname{Av}_{\ell-1}(321):
\tau_1<\tau_2<\cdots<\tau_{y-1}
\right\}.
\end{align*}
Indeed, we delete the entry $1$ from $\pi$ and subtract $1$ from every
remaining entry. The resulting permutation
$\tau\in\operatorname{Av}_{\ell-1}(321)$ satisfies
$\tau_1<\tau_2<\cdots<\tau_{y-1}$.
Conversely, given such a permutation $\tau$, add $1$ to every entry
and insert $1$ in position $y$. By Lemma~\ref{321_de},
$$\#\{\pi\in\text{Av}_\ell(321):\text{Pos}_\pi(1)=y\}=\frac y\ell\binom{2\ell-y-1}{\ell-1}.$$
Thus, all possible cross-family Wilf equivalences must arise from integer
solutions to the equation
\begin{equation}\label{eq:catalan-ballot}
C_{x-1}C_{\ell-x}
=
\frac{y}{\ell}
\binom{2\ell-y-1}{\ell-1}.
\end{equation}

We first consider the simple case $x=1$. 
\begin{thm}\label{thm:x1-versus-321}
Let $\ell\ge 3$. For every $3\le y\le\ell$, the pairs
$(213,P_{\ell,1})$ and $(321,P_{\ell,y})$
are not Wilf-equivalent.
\end{thm}

\begin{proof}
By Theorem~\ref{231,Pl1} and Theorem~\ref{thm_231_method_1}, $(213,P_{\ell,\ell})$ and $(231,P_{\ell,2})$ are Wilf-equivalent. By Theorem~\ref{thm:Phi-Psi-preserve-POP}, $(321,P_{\ell,y})$ is Wilf-equivalent to $(231,P_{\ell,y})$. And $s_\ell(213,P_{\ell,\ell})=s_\ell(213,P_{\ell,1})$ by the proof of Theorem~\ref{213_no}. Then by Theorem~\ref{231_inj}, we have
\begin{equation*}
s_\ell(213,P_{\ell,1})<s_\ell(321,P_{\ell,3})<\cdots<s_\ell(321,P_{\ell,\ell}).
\end{equation*}
Therefore $(213,P_{\ell,1})$ and $(321,P_{\ell,y})$ are not Wilf-equivalent.
\end{proof}
Having dealt with the case $x=1$, 
we reindex by replacing the original $x$ with $x+1$, 
so that $C_{x-1}C_{\ell-x}$ becomes $C_xC_{\ell-x-1}$, where $1\le x\le\ell-1$.
It remains to consider the positive integer solutions of equation 
\begin{equation}\label{move}
C_{x}C_{\ell-x-1}
=
\frac{y}{\ell}
\binom{2\ell-y-1}{\ell-1},
\end{equation}
in the range $1\le x\le \ell-1$ and $3\le y\le \ell$.
We first analyze some properties of equation \eqref{move}.

Set $a=\min\{x,\ell-1-x\}$, then \eqref{move} 
implies
\begin{align}\label{Aim}
   C_aC_{\ell-1-a}=\dfrac y\ell\binom{2\ell-y-1}{\ell-1}. 
\end{align}
Here $a\le\frac{\ell-1}2$. 
Let
\begin{align*}
T_{\ell,y}=\dfrac y\ell\binom{2\ell-y-1}{\ell-1},\quad F_\ell(a,y)=\dfrac{C_aC_{\ell-1-a}}{T_{\ell,y}}.    
\end{align*}
Thus a solution is exactly the assertion $F_\ell(a,y)=1$.
\begin{lem}\label{F_decresing}
Fix $a\ge0$ and $y\ge3$. For all admissible $\ell$, the sequence $F_\ell(a,y)$
is strictly decreasing in $\ell$, and\[
\lim_{\ell\to+\infty}F_\ell(a,y)=D(a,y)=\dfrac{C_a2^{y-2a-1}}{y}.
\]
Consequently, a solution of \eqref{Aim} must satisfy $D(a,y)<1$.
\end{lem}
\begin{proof}
By definition, we obtain
\begin{align*}
\dfrac{F_{\ell+1}(a,y)}{F_\ell(a,y)}-1=-\dfrac{6\ell a+\ell y^2-3\ell y+2\ell-ay^2-3ay+4a+y^2-3y+2}{(\ell-a+1)(2\ell-y)(2\ell-y+1)}.    
\end{align*}
Let $N_{\ell,a,y}$ denote the numerator of the right-hand side. We have
\begin{align*}
    N_{\ell,a,y}
=
(\ell+1)(y-1)(y-2)
+
a(6\ell-y^2-3y+4).
\end{align*}
It remains to show that $N_{\ell,a,y}>0$. Set
$B=6\ell-y^2-3y+4$.
If $B\ge0$, the conclusion is immediate.
Suppose now that $B<0$. Since $0\le a\le\frac{\ell-1}{2}$, we have $aB\ge\frac{\ell-1}{2}B$.
For $t\ge0$, write $\ell=y+t,$
by $t\ge0$ and $y\ge3$,
\begin{align*}
N_{\ell,a,y}
&\ge
(\ell+1)(y-1)(y-2)
+
\frac{\ell-1}{2}
(6\ell-y^2-3y+4)\ge
\frac{
6t^2+(y^2+3y+2)t+y^3-y
}{2}>0.
\end{align*}
Thus $N_{\ell,a,y}>0$ in all cases.
Therefore
$F_{\ell+1}(a,y)<F_\ell(a,y),$
so the sequence $F_\ell(a,y)$ is strictly decreasing.

We next calculate its limit. 
\begin{align*}
F_\ell(a,y)
=
\frac{C_aC_{\ell-1-a}}
{\dfrac{y}{\ell}
 \binom{2\ell-y-1}{\ell-1}}=
\frac{C_a}{y}\,
\frac{\ell}{\ell-a}\,
\frac{
\binom{2\ell-2a-2}{\ell-a-1}
}{
\binom{2\ell-y-1}{\ell-1}
}.
\end{align*}
The quotient of binomial coefficients can be written as
\begin{align*}
\frac{
\binom{2\ell-2a-2}{\ell-a-1}
}{
\binom{2\ell-y-1}{\ell-1}
}
={}&
\frac{(2\ell-2a-2)!}{(2\ell-y-1)!}
\frac{(\ell-1)!}{(\ell-a-1)!}
\frac{(\ell-y)!}{(\ell-a-1)!}.
\end{align*}
Since $a$ and $y$ are fixed, as $\ell\to\infty$ we have
\begin{align*}
\lim_{\ell\to\infty}F_\ell(a,y)
=
\frac{C_a}{y}\,2^{\,y-2a-1}.
\end{align*}
Hence the existence of an admissible integer solution necessarily
implies $D(a,y)<1$.
\end{proof}

Next, we use a Python program, given in Appendix~\ref{code:equation}, to verify the cases for
$4\leq\ell\leq3{,}273$. The computation uses Lemma~\ref{F_decresing}
to restrict the possible values that need to be checked. It shows that,
for $4\leq\ell\leq3{,}273$, $2\leq x\leq\ell$, and
$3\leq y\leq\ell$, the only integer solutions of the equation (\ref{eq:catalan-ballot}) are $(\ell,x,y)\in\{(5,3,4), (6,2,4),(6,5,4)\}$.

\begin{prop}\label{lem:exceptional-triples}
None of the three exceptional triples $(\ell,x,y)\in\{(5,3,4),(6,2,4),(6,5,4)\}$
gives a Wilf equivalence.
\end{prop}

\begin{proof}
We distinguish the corresponding pairs by comparing their avoidance
numbers at $n=\ell+1$.
For the pairs involving $213$, the formula established previously gives
\begin{align}
s_{\ell+1}(213,P_{\ell,x})
={}&
C_\ell
+\sum_{r=0}^{x-2}C_rC_{\ell-r}
+\sum_{r=1}^{\ell-x-1}C_rC_{\ell-r}
-C_xC_{\ell-x}.
\label{eq:213-l-plus-one}
\end{align}
Substituting $(\ell,x)=(5,3)$ into
\eqref{eq:213-l-plus-one}, we obtain
$s_6(213,P_{5,3})=102$.
Similarly,
$s_7(213,P_{6,2})=331$ and 
$s_7(213,P_{6,5})=317$.

It remains to calculate the corresponding avoidance numbers for
$(321,P_{\ell,4})$. By the Wilf equivalence established above, for
$2\leq j\leq\ell$ the pairs
$(321,P_{\ell,j})$ and $(231,P_{\ell,j})$ are Wilf-equivalent and $s_n(231,P_{\ell,j})
=
\left|\mathcal C_n(j-1,\ell-j)\right|$. Therefore,
\begin{equation}\label{eq:321-catalan-sequences}
s_n(321,P_{\ell,4})
=
\left|\mathcal C_n(3,\ell-4)\right|.
\end{equation}
The two relevant cardinalities can be computed directly.
Let $d_i(t)$ denote the number of admissible weakly increasing prefixes
$(q_1,\ldots,q_i)$ ending with $q_i=t$. The admissibility conditions are
$i\leq t\leq n$ and
$t\notin\{i+a,i+a+1,\ldots,n-b\}$.
Thus
\[
d_1(t)=
\begin{cases}
1,&t\text{ is admissible},\\
0,&\text{otherwise},
\end{cases}
\]
and, for $i\ge2$,
\begin{equation}\label{eq:d-recurrence}
d_i(t)
=
\begin{cases}
\displaystyle\sum_{s\le t}d_{i-1}(s),
   &t\text{ is admissible},\\[2mm]
0,&\text{otherwise}.
\end{cases}
\end{equation}
Since the last coordinate must be $n$, we have
$|\mathcal C_n(a,b)|=d_n(n)$.
For $\mathcal C_6(3,1)$, recurrence
\eqref{eq:d-recurrence} gives
\[
\begin{array}{c|rrrrrr}
i\backslash t&1&2&3&4&5&6\\ \hline
1&1&1&1&0&0&1\\
2&0&2&3&3&0&4\\
3&0&0&5&8&8&12\\
4&0&0&0&13&21&33\\
5&0&0&0&0&34&67\\
6&0&0&0&0&0&101
\end{array}
\]
and hence $|\mathcal C_6(3,1)|=101$.
For $\mathcal C_7(3,2)$, the same recurrence gives
$|\mathcal C_7(3,2)|=319$.
It follows from \eqref{eq:321-catalan-sequences} that
$s_6(321,P_{5,4})=101$
and $s_7(321,P_{6,4})=319$.
Consequently, none of these three exceptional triples gives a Wilf equivalence.
\end{proof}

It remains only to consider the case $\ell\ge 3{,}274$.
In this range, it suffices to determine whether equation \eqref{move}
admits any positive integer solutions. In fact, we have the following theorem.
\begin{thm}\label{Thm-last}
There are no positive integer solutions of
\begin{align*}
    \dfrac\ell{(x+1)(\ell-x)}\binom{2x}x\binom{2\ell-2x-2}{\ell-x-1}=y\binom{2\ell-y-1}{\ell-1}
\end{align*}
with $\ell\ge3{,}274, 1\le x\le\ell-1$, and $3\le y\le\ell$.
\end{thm}
Indeed, it suffices to consider the transformed equation \eqref{Aim}.
\begin{lem}\label{lem:no-solutions-large-l}
No solution of \eqref{Aim} exists with $0\le a\le29$ and $\ell\ge 3{,}274$.
\end{lem}
\begin{proof}
By Lemma~\ref{F_decresing} we only need to consider the pairs $(a,y)$ satisfying $D(a,y)<1$.
Since $D(a,y)<1
\quad\Longleftrightarrow\quad
C_a2^y<y\,2^{2a+1}$,
we seek an upper bound for $y$. Using the fact that $\binom{2a}{a}$ is the largest binomial
coefficient in the expansion of \((1+1)^{2a}\), we have
$4^a\le
(2a+1)\binom{2a}{a}$.
Hence
\begin{align*}
    C_a
=
\frac{1}{a+1}\binom{2a}{a}
\ge
\frac{4^a}{(a+1)(2a+1)}.
\end{align*}
Combining this inequality with \(C_a2^y<y2^{2a+1}\), we obtain
\begin{align*}
\frac{4^a}{(a+1)(2a+1)}\,2^y
<
y\,2^{2a+1},
\end{align*}
which reduces to
$2^{y-1}<y(a+1)(2a+1)$.
For $0\le a\le29$, we have
$2^{y-1}<1770y$.
Thus the necessary condition $D(a,y)<1$ forces $3\le y\le15$.

It remains only to consider the finite range
$0\leq a\leq29$, $3\leq y\leq15$, and $D(a,y)<1$. For these pairs, a
Python computation given in Appendix~\ref{app:code} shows that
$F_{3{,}274}(a,y)<1$. By Lemma~\ref{F_decresing},
$F_\ell(a,y)$ is strictly decreasing in $\ell$. Hence, for every
$\ell\ge 3{,}274$,
$F_\ell(a,y)\leq F_{3{,}274}(a,y)<1$.
Thus, $F_\ell(a,y)\neq1$, and therefore the original equation has no
integer solutions in the stated range.
\end{proof}

Next, we analyze the solutions in the remaining range of parameters.
\begin{lem}\label{lem-bound-a}
Any solution of \eqref{Aim} with $\ell \ge 3{,}274$ has
$a\ge30$, and $y\le\frac a2$.
\end{lem}

\begin{proof}
By Lemma~\ref{lem:no-solutions-large-l}, any solution with $\ell\ge 3{,}274$ must satisfy
$a\ge30$, and the existence of a solution implies
$\frac{2^{y-1}}{y}<(a+1)(2a+1)$.
We prove \(y\le a/2\) by contradiction. Suppose instead that $y>\frac{a}{2}$.
Since $y$ is an integer, this implies
$y\ge y_0:=\left\lfloor\frac{a}{2}\right\rfloor+1$.
We know that $\frac{2^{y-1}}{y}$ is increasing for $y\ge 2$.
Consequently,
\begin{align}
    \frac{2^{y_0-1}}{y_0}
<
(a+1)(2a+1).\label{eq:y0-necessary}
\end{align}
We shall show that, for every $a\ge30$, the reverse strict inequality
holds.
We prove by induction in steps of two. 
We can check the reverse strict inequality holds for $a=30$ and $a=31$.
Now suppose the reverse strict inequality holds for some $a\ge30$.
When $a$ is replaced by $a+2$,
the factor by which the left-hand side of \eqref{eq:y0-necessary} increases is $\frac{2y_0}{y_0+1}$.
Since $a\ge30$, we have  $\frac{2y_0}{y_0+1}
\ge
\frac{32}{17}$.
On the other hand, for $a\ge 30$ the factor by which the right-hand side increases is
$\frac{(a+3)(2a+5)}{(a+1)(2a+1)}<
\frac{32}{17}$.
This contradicts \eqref{eq:y0-necessary}, which proves the lemma.
\end{proof}
To prove Theorem~\ref{Thm-last}, we also need the following results from number theory.
\begin{lem}[{\cite[Theorem 1.9]{Dusart1998}}]\label{lem-number-1}
For every real number $X\ge3{,}275$, there is a prime $p$ such that
\begin{align*}
X<p\le X\left(1+\dfrac1{2(\ln X)^2} \right).    
\end{align*}
\end{lem}

The following result is an immediate consequence of Theorem~2 of Nair and Shorey \cite{NairShorey2016}.
\begin{lem}\label{lem-number-2}
Let $N,m$ be positive integers. If $N>4m$, $m>3$, and $N+m>150$,
then the product $N(N+1)\cdots(N+m-1)$ possesses a prime divisor greater than $4.42m$.
\end{lem}

With the above preparations in place, we are now ready to prove Theorem~\ref{Thm-last}. Using Lemmas~\ref{lem-number-1} and~\ref{lem-number-2}, 
we can directly analyze the existence of solutions to equation \eqref{Aim}.

\begin{proof}[Proof of Theorem~\ref{Thm-last}]
Suppose, to the contrary, that
$(\ell,a,y)$ is a solution of \eqref{Aim} with
$\ell\ge 3{,}274$. 
By Lemma~\ref{lem-bound-a}, we have $a\ge30$, and $y\le\frac{a}{2}$.

Expanding $C_{\ell-a-1}$ in \eqref{Aim}, we have
\begin{align*}
\frac{\ell C_a}{\ell-a}
=
y\,
\frac{\binom{2\ell-y-1}{\ell-1}}
     {\binom{2\ell-2a-2}{\ell-a-1}}.
\end{align*}
Equivalently,
\begin{equation}\label{eq:product-first}
\ell C_a
\prod_{i=1}^{a}(\ell-i)
\prod_{j=y}^{a-1}(\ell-j)
=
y\prod_{k=y+1}^{2a+1}(2\ell-k).
\end{equation}

We next cancel the factors corresponding to even values of $k$ on
the right-hand side. If $k=2i$, then
$2\ell-k=2(\ell-i)$.
The even values of $k$ in the interval
$y+1\le k\le2a+1$ correspond precisely to
$\left\lfloor\frac{y}{2}\right\rfloor+1
\le i\le a$.
After cancelling the associated factors $\ell-i$ from the first
product on the left-hand side of \eqref{eq:product-first}, we obtain
\begin{equation}\label{eq:product-reduced}
\ell C_a
\prod_{i=1}^{\lfloor y/2\rfloor}(\ell-i)
\prod_{j=y}^{a-1}(\ell-j)
=
y\,2^{a-\lfloor y/2\rfloor}
\prod_{r=\lceil y/2\rceil}^{a}(2\ell-2r-1).
\end{equation}
Define $m=a-y$ and $N=\ell-a+1$.
Then
\begin{align*}
    \prod_{j=y}^{a-1}(\ell-j)
=
N(N+1)\cdots(N+m-1),
\end{align*}
which is a product of $m$ consecutive positive integers. 
We know that $ m=a-y\ge\frac{a}{2}\ge15$.
We distinguish two cases.

\noindent \textbf{Case 1:} $N>4m$

Since $N+m=\ell-y+1$
and $y\le\frac{a}{2}\le\frac{\ell}{2}$,
we have
$N+m\ge\frac{\ell}{2}+1>150$.
Therefore, Lemma~\ref{lem-number-2} guarantees that the product
$N(N+1)\cdots(N+m-1)$
has a prime divisor $p$ satisfying $p>4.42m$.
Then there exists some $j\in\{y,\ldots,a-1\}$ such that $p\mid\ell-j$.

We claim that $p$ does not divide the right-hand side of
\eqref{eq:product-reduced}.
Since $m\ge a/2$ and $a\ge30$, we have $p>4.42m\ge2.21a>2a+1>y$, and $p$ is odd, $p\nmid y2^{a-\lfloor y/2\rfloor}$. If $p$ divided the remaining
product on the right, then for some
$\left\lceil\frac{y}{2}\right\rceil\le r\le a$
we would have $p\mid 2\ell-(2r+1)$.
Combining this with $p\mid\ell-j$ gives
\begin{align*}
p\mid
2(\ell-j)-\bigl(2\ell-(2r+1)\bigr)
=
2r+1-2j.
\end{align*}
However, $2r+1-2j$ is a nonzero odd integer and $0<|2r+1-2j|\le2a+1<p$,
which is impossible. Thus $p$ divides the left-hand side of
\eqref{eq:product-reduced} but not its right-hand side, a
contradiction.

\noindent \textbf{Case 2:} $N\le4m$

Set $A=2\ell-(2a+1)$, $B=2\ell-
\left(2\left\lceil\frac{y}{2}\right\rceil+1\right)$.
Then the final product on the right-hand side of
\eqref{eq:product-reduced} is
\begin{align*}
\prod_{r=\lceil y/2\rceil}^{a}(2\ell-2r-1)
=
A(A+2)(A+4)\cdots B.
\end{align*}
Since $a\le\frac{\ell-1}{2}$, 
we have $A=2\ell-(2a+1)\ge\ell$ and $A\ge2a+1$.
Moreover, $A=2\ell-2a-1$ is odd, so we have $A\ge 3{,}275$.
Also, since $N\le4m$ gives $\ell-a+1\le4(a-y)$,
so $\ell\le5a-4y-1$.
Consequently, $A=2\ell-2a-1
\le8a-8y-3<8a$.
On the other hand,
$B-A=2a-2\left\lceil\frac{y}{2}\right\rceil\ge\frac{3a}{2}-1$.

By Lemma~\ref{lem-number-1}, since $A\ge3{,}275$, there exists a prime number $q$
such that
\begin{align*}
A<q\le
A\left(1+\frac{1}{2(\ln A)^2}\right).    
\end{align*}
Furthermore,
\begin{align*}
\frac{A}{2(\ln A)^2}
<
\frac{8a}{2(\ln3{,}275)^2}<0.062a<\frac{3a}{2}-1\le B-A.
\end{align*}
It follows that $A<q\le B$.
So $q$ is one of the factors $A,A+2,\ldots,B$.
Therefore $q$ divides the right-hand side of
\eqref{eq:product-reduced}.

We now show that \(q\) does not divide its left-hand side. First,
$q>A\ge2a+1$,
so $q\nmid C_a$.
The other factors on the left-hand side have the form $\ell-i$,
where
\begin{align*}
i\in
\left\{0,1,\ldots,\left\lfloor\frac{y}{2}\right\rfloor\right\}
\cup
\{y,y+1,\ldots,a-1\};    
\end{align*}
Since $q$ is a factor of the right-hand product, there exists
$\left\lceil\frac{y}{2}\right\rceil\le r\le a$ such that $q=2\ell-(2r+1)$.
If \(q\mid\ell-i\) for one of the above values of \(i\), then
$q\mid
2(\ell-i)-(2\ell-(2r+1))
=
2r+1-2i$.
But $2r+1-2i$ is a nonzero odd integer satisfying
$0<|2r+1-2i|\le2a+1<q$,
again a contradiction. Hence $q$ does not divide any factor on
the left-hand side of \eqref{eq:product-reduced}.
Both cases lead to contradictions. Therefore
\eqref{Aim} has no solutions with
$\ell\ge 3{,}274$.
\end{proof}
Therefore, $(213,P_{\ell,x})$ and $(321,P_{\ell,y})$ are not
Wilf-equivalent, as claimed. This completes the classification of all
Wilf-equivalence classes.

\section{A distribution for quadrant marked mesh patterns}\label{S4}
Sections~\ref{S2} and~\ref{S3} were devoted to the classification of
Wilf equivalences for permutations simultaneously avoiding a pattern of
length three and $P_{\ell,x}$. In this section, we turn to a distribution
problem. In particular, we prove the following identity, conjectured by
Qiu and Remmel \cite{QiuRemmel2018}.

\begin{thm}\label{thm-distribution}
For all $k\ge1$, we have
\begin{align*}
Q_{132}^{(0,k,\emptyset,0)}(t,q)=Q_{132}^{(1,k-1,\emptyset,0)}(t,q).  
\end{align*}
\end{thm}
\begin{proof}
Let $A_k=Q_{132}^{(0,k,\emptyset,0)}(t,q)$ and $B_k=Q_{132}^{(1,k-1,\emptyset,0)}(t,q)$. 
Theorem~13 in \cite{QiuRemmel2018} states that
\begin{align*}
A_k=\frac1{1-tA_0}\left(1+t\sum_{j=0}^{k-2}C_jt^j(A_{k-1-j}-A_0)\right).    
\end{align*}

Let $C(z)$ be the generating function for the Catalan numbers and let
$P(z)=\sum_{k\geq1}A_kz^{k-1}$.
Multiplying both sides by $z^{k-1}$ and interchanging the order of summation,
we obtain
\begin{align*}
P(z)=\frac1{(1-tA_0)(1-z)}
-\frac{A_0tzC(tz)}{(1-tA_0)(1-z)}
+\frac{ztC(tz)P(z)}{1-tA_0}.
\end{align*}
Hence,
\begin{align*}
P(z)=\frac{1-A_0ztC(zt)}
{(1-tA_0-ztC(tz))(1-z)}.
\end{align*}
On the other hand,
\begin{align}
&\sum_{k\geq2}\left(
1+t\sum_{j=0}^{k-3}C_jt^jA_{k-1-j}
+\frac t{1-tA_0}\left(
A_{k-1}-\sum_{j=0}^{k-3}C_jt^j
\right)\right)z^{k-1}\notag\\
&=\frac z{1-z}+ztC(tz)P(z)-\frac{ztC(tz)}{1-tA_0}
+\frac{ztP(z)}{1-tA_0}
-\frac{z^2tC(tz)}{(1-tA_0)(1-z)}\notag\\
&=\frac{1-A_0ztC(tz)}
{(1-tA_0-ztC(tz))(1-z)}
-\frac1{1-tA_0}
+\frac{zt(1-A_0+tA_0^2)(ztC^2(tz)-C(tz)+1)}
{(1-tA_0)(1-tA_0-ztC(tz))(1-z)}\notag\\
&=\frac{1-A_0ztC(tz)}
{(1-tA_0-ztC(tz))(1-z)}
-\frac1{1-tA_0}
=P(z)-A_1.
\end{align}
Therefore,
\begin{align*}
A_k
=
1+t\sum_{j=0}^{k-3}C_jt^jA_{k-1-j}
+tA_1\left(
A_{k-1}-\sum_{j=0}^{k-3}C_jt^j
\right).
\end{align*}
Using the corrected form of \cite[Theorem 14]{QiuRemmel2018},
with the missing factor $t$ in its last term restored, we obtain
\begin{align*}
B_k
=
1+t\sum_{j=0}^{k-3}C_jt^jB_{k-1-j}
+tB_1\left(
A_{k-1}-\sum_{j=0}^{k-3}C_jt^j
\right).
\end{align*}
By \cite[Eq.~(58)]{QiuRemmel2018}, we have $A_1=B_1$.
Comparing the two recurrences above, we then obtain
$A_k=B_k$ for all $k\geq1$ by induction.
\end{proof}

Finally, we point out a connection between Theorem~\ref{thm-distribution} and
Theorem~\ref{thm_231_method_1}. By \cite[Corollary~1]{QiuRemmel2018}, within
$\operatorname{Av}_n(132)$, replacing the third parameter $\emptyset$ by $0$
preserves the coefficients of $q^0$ and $q^1$ in the corresponding
generating functions. Hence, by taking the coefficient of $q^0$ in
Theorem~\ref{thm-distribution} and then applying reversal, we recover the
Wilf equivalence between $(231,P_{\ell,1})$ and $(231,P_{\ell,2})$
established in Theorem~\ref{thm_231_method_1}.

\section*{Acknowledgements.}
The work of the first and third authors is supported by the Fundamental Research Funds 
for the Central Universities.

\newpage
\appendix
\section{Computational verification}
\subsection{Python code verifying Equation~\eqref{eq:catalan-ballot} for $4\le \ell\le 3{,}273$}
\label{code:equation}

\begin{lstlisting}[
caption={Exact search for the integer solutions of Equation \eqref{eq:catalan-ballot} with $4\le\ell\le3{,}273$.},
label={code:equation6}
]
def catalans_up_to(n):
    C = [1] * (n + 1)
    for k in range(n):
        C[k + 1] = C[k] * 2 * (2 * k + 1) // (k + 2)
    return C

def solve_equation(max_ell=3273):
    C = catalans_up_to(max_ell)
    solutions = []
    for ell in range(4, max_ell + 1):
        max_a = (ell - 1) // 2
        y = 3
        while (
            y <= ell
            and C[max_a] * (1 << y)
            < y * (1 << (2 * max_a + 1))
        ):
            y += 1
        max_y = y - 1
        if max_y < 3:
            continue      
        R = ell * C[ell - 1]
        current_y = 2
        for y in range(3, max_y + 1):
            old_y = current_y
            R = (
                R
                * (old_y + 1)
                * (ell - old_y)
                // (
                    old_y
                    * (2 * ell - old_y - 1)
                )
            )
            current_y = y
            for a in range(0, max_a + 1):

                L = ell * C[a] * C[ell - 1 - a]
                if L == R:
                    x1 = a + 1
                    x2 = ell - a
                    if 2 <= x1 <= ell:
                        solutions.append((ell, x1, y))
                    if x2 != x1 and 2 <= x2 <= ell:
                        solutions.append((ell, x2, y))
    solutions.sort()
    return solutions
solutions = solve_equation(3273)
print("All integer solutions for 4 <= ell <= 3273 are:")
for sol in solutions:
    print(sol)
print("Total number of solutions in the specified range", len(solutions))
\end{lstlisting}

The program returns:

\begin{verbatim}
All integer solutions for 4 <= ell <= 3273 are:
(5, 3, 4)
(6, 2, 4)
(6, 5, 4)
Total number of solutions in the specified range 3
\end{verbatim}

\subsection{Python code verifying Lemma~\ref{lem:no-solutions-large-l}}\label{app:code}
	\begin{lstlisting}[
caption={Exact verification of the finite cases in Lemma~\ref{lem:no-solutions-large-l}},
label={code:lemma314}
]
try:
    from math import comb
except ImportError:
    def comb(n, k):
        if k < 0 or k > n:
            return 0
        k = min(k, n - k)
        result = 1
        for i in range(1, k + 1):
            result = result * (n - k + i) // i
        return result

def C(n: int) -> int:
    """Return the n-th Catalan number."""
    return comb(2 * n, n) // (n + 1)
L = 3274
for a in range(30):          
    for y in range(3, 16):  
        # D(a, y) < 1 is equivalent to C_a * 2^y < y * 2^(2a+1).
        if C(a) * (1 << y) < y * (1 << (2 * a + 1)):
            left = L * C(a) * C(L - 1 - a)
            right = y * comb(2 * L - y - 1, L - 1)

            assert left < right, (
                f"Verification failed for a={a}, y={y}: "
                f"left={left}, right={right}"
            )
print("All cases passed.")
\end{lstlisting}
The program returns:
\begin{verbatim}
All cases passed.
\end{verbatim}
\end{document}